\documentclass[a4paper,
               headinclude=false,
               headsepline,
               footinclude=false,
               footsepline,
               plainfootsepline,
               abstract=true,
               twoside]{scrartcl}
\usepackage[active]{srcltx} 
\usepackage{amssymb,amsmath,amsthm}
\usepackage{multicol}
\usepackage{tikz}
\usetikzlibrary{arrows,shapes}

\usepackage{scrlayer-scrpage}
\RequirePackage{hyperref}
\usepackage{charter}\providecommand{\coloneqq}{:=}
\addtokomafont{title}{\rmfamily}
\addtokomafont{section}{\rmfamily}
\addtokomafont{subsection}{\rmfamily}
\swapnumbers
\theoremstyle{plain}
\newtheorem{theo}{Theorem}[section]
\newtheorem{coro}[theo]{Corollary}
\newtheorem{lemm}[theo]{Lemma}
\newtheorem*{lemm*}{Lemma}
\newtheorem{prop}[theo]{Proposition}
\newtheorem{namet}[theo]{\myThmName}
\newtheorem*{namet*}{\myThmName}
\newenvironment{nthm}[1][\kern-.35em]{\edef\myThmName{#1}\begin{namet}}{\end{namet}}
\newenvironment*{nthm*}[1][\kern-.35em]{\edef\myThmName{#1}\begin{namet*}}{\end{namet*}}

\theoremstyle{definition}
\newtheorem{defi}[theo]{Definition}
\newtheorem*{defi*}{Definition}

\newtheorem{exam}[theo]{Example}
\newtheorem{exas}[theo]{Examples}
\newtheorem{rema}[theo]{Remark}
\newtheorem*{rema*}{Remark}
\newtheorem{rems}[theo]{Remarks}

\newtheorem{named}[theo]{\myThmName}
\newtheorem*{named*}{\myThmName}
\newenvironment{ndef}[1][\kern-.35em]{\edef\myThmName{#1}\begin{named}}{\end{named}}

\newcounter{rememberEnumi}
\newcommand{\SaveEnumi}{\global\setcounter{rememberEnumi}{\value{enumi}}}
\newcommand{\RecallEnumi}{\setcounter{enumi}{\therememberEnumi}}
\let\setminus\smallsetminus
\newcommand{\HH}{\ensuremath{\mathbb{H}}}
\newcommand{\OO}{\ensuremath{\mathbb{O}}}
\newcommand{\ZZ}{\ensuremath{\mathbb{Z}}}
\newcommand{\id}{\ensuremath{\mathrm{id}}}
\let\setminus\smallsetminus
\makeatletter
\newcommand{\gal}[1]{ {\mathchoice
    {\vbox
    {\m@th \ialign {##\crcr \noalign {\kern 1\p@ }\kern 1\p@ \hrulefill \crcr
        \noalign {\kern 1\p@ \nointerlineskip }%
        $\hfil \displaystyle {#1}\hfil $\crcr }}}
    {\vbox
    {\m@th \ialign {##\crcr \noalign {\kern 1\p@ }\kern 1\p@ \hrulefill \crcr
        \noalign {\kern 1\p@ \nointerlineskip }%
        $\hfil \textstyle {#1} $\crcr }}}
    {\vbox
    {\m@th \ialign {##\crcr \noalign {\kern 1\p@ }\kern 1\p@ \hrulefill \crcr
        \noalign {\kern 1\p@ \nointerlineskip }%
        $\hfil \scriptstyle {#1}\hfil $\crcr }}}
    {\vbox
    {\m@th \ialign {##\crcr \noalign {\kern 1\p@ }\kern 1\p@ \hrulefill \crcr
        \noalign {\kern 1\p@ \nointerlineskip }%
        $\hfil \scriptscriptstyle {#1}\hfil $\crcr }}}%
    }}
\makeatother
\newcommand{\N}[1][]{{\mathrm{N}}_{#1}}
\newcommand{\tr}[1][]{{\mathrm{T}}_{#1}}
\newcommand{\set}[2]{\left\{{#1}\left|\vphantom{#1#2\strut}\right.\, 
                    {#2}\right\}}
\newcommand{\smallset}[2]{\{{#1}\left|\vphantom{}\right.\, 
                    {#2}\}}
\newcommand{\C}[2][]{\operatorname{C}_{#1}(#2)}
\newcommand{\Aut}[2][]{\operatorname{Aut}_{#1}(#2)}

\newcommand{\Char}{\operatorname{char}}
\newcommand{\GL}[2][]{\mathrm{GL}_{#1}{(#2)}}
\newcommand{\SL}[2][]{\mathrm{SL}_{#1}{(#2)}}
\lohead[]{\footnotesize\normalfont\rmfamily \MYtitle}
\lehead[]{\footnotesize\normalfont\rmfamily \MYauthor}
\rehead[]{\footnotesize\normalfont\rmfamily \MYtitle}
\rohead[]{\footnotesize\normalfont\rmfamily \MYauthor}
\lefoot[\pagemark]{\pagemark}
\cefoot[]%
       {}
\refoot[\tiny\copyright~{\normalfont\ttfamily
  by~the authors
}]{\tiny\copyright~{\normalfont\ttfamily by~the authors}}
\lofoot[\tiny\copyright~{\normalfont\ttfamily
  by~the authors
}]{\tiny\copyright~{\normalfont\ttfamily by~the authors}} 
\cofoot[]%
       {}
\rofoot[\pagemark]{\pagemark}
\title{Subalgebras of Octonion Algebras} %
\author{%
  {Norbert Knarr,} 
  {Markus J.~Stroppel}}
  \let\MYauthor\shortauthor 
  \let\MYtitle\shorttitle
\newcommand{\keywords}[1]{\par\noindent{\normalfont\bfseries Keywords: }#1}
\newcommand{\subjclass}[1]{\par\noindent{\normalfont\bfseries Mathematics Subject
    Classification (MSC 2000): }#1}
\newcommand{\MSC}[1]{\href{http://www.ams.org/mathscinet/search/mscbrowse.html?code=#1}{#1}}
\begin{document}
\maketitle
\begin{abstract}\noindent %
  For an arbitrary octonion algebra, we determine all subalgebras. It
  turns out that every subalgebra of dimension less than four is
  associative, while every subalgebra of dimension greater than four
  is not associative. In any split octonion algebra, there are both
  associative and non-associative subalgebras of dimension four.
  Except for one-dimensional subalgebras spanned by idempotents, any
  two isomorphic subalgebras are in the same orbit under
  automorphisms. %
\end{abstract}

\section*{Introduction}

Octonion algebras are the most complicated examples of composition
algebras. %
Composition algebras are (not necessarily associative) algebras
endowed with a multiplicative quadratic form such that the associated
polar form is not degenerate. As a tool for the investigation of
quadratic forms, they have their roots in number
theory. Four-dimensional composition algebras (also called quaternion
algebras) are ubiquitous in the theory of quadratic forms.  
In the present paper, we are mainly concerned with composition
algebras of dimension~$8$, also called \emph{octonion algebras} (or
octaves).

It is known (see~\cite[1.6.2]{MR1763974}) that composition algebras
only occur in dimension $d\in\{1,2,4,8\}$. A composition algebra is
called \emph{split} if the quadratic form is isotropic. Over any
 field~$F$, there are split composition algebras of
dimension~$2$ (isomorphic to $F\times F$ with the quadratic form
mapping $(x,y)$ to~$xy$), of dimension~$4$ (isomorphic to
$F^{2\times2}$, the quadratic form is the determinant) and of
dimension~$8$ (obtained by doubling the split quaternion
algebra~$F^{2\times2}$, see~\ref{propertiesCA}\ref{doubling} below).
These split composition algebras are determined uniquely by their
dimension and the underlying field~$F$ (see~\cite[1.8.1]{MR1763974}).

The existence of non-split composition algebras depends in an
essential way on properties of the ground field~$F$. In particular, a
non-split composition algebra of dimension~$2$ over~$F$ is a separable
quadratic extension field. A non-split quaternion algebra exists if,
and only if, there is a separable quadratic extension and an
anisotropic hermitian form in two variables over that extension
(cp.~\cite{MR1681303}).  In particular, every composition algebra of
dimension greater than two over a finite field is split.  A quaternion
field (i.e., a non-split quaternion algebra) over~$F$ exists if, and
only if, there is a non-singular quadratic form on~$F^3$ such that the
polar form is not zero (see~\cite[Th.\,3.1.1]{MR3761133}). A necessary
(but not sufficient, see~\cite[Th.\,3.1.6]{MR3761133}) condition for
the existence of a non-split octonion algebra over~$F$ is of course
the existence of a non-singular quadratic form on~$F^8$.

Octonion algebras play a role in group theory: For instance, the
automorphism group of an octonion algebra is a group of type
$\mathrm{G}_2$, the action of the hyperplane orthogonal to~$1$ gives
the representation of that group on a vector space of dimension~$7$. %
Several important subgroups of that group can conveniently be seen as
stabilizers of certain elements or subspaces in the octonion algebra
(see~\cite{MR3871471}), %
and various exceptional isomorphisms between classical groups can be
understood using representations as such stabilizers.

We also remark that non-split octonion algebras of dimension~$8$
coordinatize Moufang planes that are not desarguesian, and that every
non-desarguesian Moufang plane is obtained in this way
(\cite{0007.07205}, \cite[Satz\,50, p.\,106 and Abschnitt\,6,
pp.\,157--186]{MR0073211}, \cite[Ch.\,19]{MR1938841}). %
Split octonion algebras are used to construct generalized hexagons
(namely, the split Cayley hexagons, see~\cite[p.\,207]{MR0143075}).

In the present paper, we study the subalgebras of octonion algebras.
The split case is of special interest; we find many types of
subalgebras in split octonion algebras that do not have counterparts
in non-split algebras. %
We also prove that subalgebras that are isomorphic as abstract
algebras are in the same orbit under automorphisms of the octonion
algebra (except if they are one-dimensional, and spanned by~$1$ and an
idempotent $p\ne1$, respectively). %
See Figure~\ref{fig:lattice} in Section~\ref{sec:maximal} below for a
graph showing the relative position of (orbits of) subalgebras. %
As a byproduct, we obtain that every subalgebra of dimension less
than~$4$ in a composition algebra is associative
(see~\ref{assocResult}). In split octonion algebras, there do occur
subalgebras of dimension~$4$ that are not associative
(see~\ref{fourdimNotAssoc}), and subalgebras of dimensions~$5$
and~$6$, respectively (see~\ref{Exasleftidealconstruction}); no
subalgebra of dimension greater than~$4$ is associative
(see~\ref{assocResult}).

Using our list of subalgebras, we identify the commutative
subalgebras, see~\ref{commResult}. It turns out that every
commutative subalgebra is associative. %

Maximal subalgebras of octonion algebras have been studied
in~\cite{MR349771}, a gap in that paper was noted in~\cite{MR2965899}
and closed in~\cite{MR3042630}. Unary subalgebras (i.e. subalgebras
containing the neutral element of multiplication) of finite octonion
algebras have been determined in~\cite{MR2726555}. %
Our present treatment is independent of those results. 

\section{Composition algebras}

\begin{defi}
  Let\/~$F$ be a field, and let\/ $(A,+,\cdot)$ be any (not
  necessarily unary, not necessarily associative) algebra over~$F$. %
  A quadratic form $q\colon A\to F$ is called \emph{multiplicative} if
  $q(x\cdot y)=q(x)q(y)$ holds for all $x,y\in A$. %
\end{defi}

\begin{nthm}[Lemma: multiplicative forms]\label{isotope}
  Let\/ $(A,+,\cdot)$ be any algebra over~$F$. %
  Assume that $a,b\in A$ are such that the maps
  $\rho_a\colon A\to A\colon x\mapsto xa$ and
  $\lambda_b\colon A\to A\colon x\mapsto bx$ are bijections, and
  define the multiplication~$*$ on~$A$ by
  $x*y \coloneqq \rho_a^{-1}(x)\cdot\lambda_b^{-1}(y)$. %
  \begin{enumerate}
  \item The element $b\cdot a$ is a neutral element for the multiplication~$*$.
  \item The triplet $(\lambda_b^{-1}|\id|\rho_a^{-1})$ is an isotopism from
    $(A,+,*)$ onto $(A,+,\cdot)$.
  \item Assume that there is a non-zero multiplicative quadratic
    form~$q$ on $(A,+,\cdot)$. \\%
    Then $s\coloneqq q(b\cdot a)$ is not zero, and $sq(x*y) = q(x)q(y)$
    holds for all $x,y\in A$. \\%
    The form $s^{-1}q$ is thus a non-zero multiplicative form on
    $(A,+,*)$.
  \end{enumerate}
\end{nthm}
\begin{proof}
  We compute $(b\cdot a)*y = \rho_a^{-1}(b\cdot a)\cdot\lambda_b^{-1}(y) %
  = b\cdot\lambda_b^{-1}(y) = \lambda_b(\lambda_b^{-1}(y)) = y$ %
  and $x*(b\cdot a) = \rho_a^{-1}(x)\cdot\lambda_b^{-1}(b\cdot a) %
  = \rho_a^{-1}(x)\cdot a = \rho_a(\rho_a^{-1}(x)) = x$. %
  This shows that~$b\cdot a$ is a neutral element for the
  multiplication~$*$.

  For $x,y\in A$, we note
  $\id(x*y) = x*y = \rho_a^{-1}(x)\cdot\lambda_b^{-1}(y)$. %
  This is the defining property of an isotopism from $(A,+,*)$ onto
  $(A,+,\cdot)$, cp.~\cite[Section\,1]{MR3019290}.  

  Pick $u\in A$ such that $q(u)\ne0$, and let\/
  $v\coloneqq\lambda_b^{-1}(\rho_a^{-1}(u))$. Then
  $0 \ne q(u) = q((b\cdot v)\cdot a) = q(b\cdot v)q(a) = q(b)q(v)q(a)
  = q(b)q(a)q(v) = q(b\cdot a)q(v)$ yields $q(b\cdot a)\ne0$. We
  compute %
  \(%
  sq({x*y}) = %
  sq\left(\rho_a^{-1}(x)\cdot\lambda_b^{-1}(y)\right) =
  \left(q(a)q(b)\strut\right)
  q\left(\rho_a^{-1}(x)\right)q\left(\lambda_b^{-1}(y)\right) %
  = %
  q\left(\rho_a^{-1}(x)\right)q(a)q(b)q\left(\lambda_b^{-1}(y)\right)
  = q\left(a\cdot \rho_a^{-1}(x)\right) q\left(b\cdot
    \lambda_b^{-1}(y)\right) %
  = q(x)q(y) %
  \). %
  This implies the last assertion. 
\end{proof}

\begin{defi}
A \emph{composition algebra} over~$F$ is a (not necessarily
associative) algebra~$C$ over~$F$ with a multiplicative quadratic form
$\N\coloneqq\N[C]\colon C\to F$ (called the \emph{norm form}) such
that the polar form of~$\N$ is not degenerate. %
This polar form will be written as
$(x|y) \coloneqq \N(x+y)-\N(x)-\N(y)$. %
We also assume that the algebra contains a neutral element for its
multiplication, denoted by~$1$. %
(However, see~\ref{isotope} for the case of isotopes without~$1$.)

As usual, the ground field~$F$ is embedded as~$F1$ in~$C$. %
The composition algebra is called \emph{split} if it contains divisors
of zero (i.e., non-zero elements of norm zero,
see~\ref{propertiesCA}\ref{inverse} below).
If no confusion is expected we write multiplication as juxtaposition. 
\end{defi}

The first chapter of~\cite{MR1763974} gives a comprehensive
introduction into composition algebras over arbitrary fields,
including the characteristic two case. %
We collect the basic facts that we need in the present paper (for
proofs, consult~\cite{MR1763974}): %

\begin{nthm}[Properties of composition algebras]\label{propertiesCA}
  Let\/ $C$ be a composition algebra over~$F$.
  \begin{enumerate}
  \item\label{kappaReflection}%
    The map
    $\kappa\colon C\to C\colon x\mapsto \gal{x} \coloneqq (x|1)1-x$ is
    an involutory anti-automorphism, called the \emph{standard
      involution} of~$C$. %
  \item\label{recoverNorm}%
    The norm and its polar form are recovered from the standard
    involution as $\N(x)=x\gal{x} = \gal{x}x$ and
    $(x|y) = x\gal{y}+y\gal{x}$. In particular, we have the hyperplane
    $1^\perp = \set{x\in C}{\gal{x}=-x}$ of %
    \emph{pure elements}.
  \item\label{adjoint}%
    For all $c,x,y\in C$, we have $(cx|y)=(x|\gal{c}y)$ and
    $(xc|y)=(x|y\gal{c})$, see~{\upshape\cite[1.3.2]{MR1763974}}.
  \item\label{MoufangId}%
    In general, the multiplication is not associative, but weak
    versions of the associative law are still valid; among them
    \emph{Moufang's identities}~{\upshape\cite[1.4.1, 1.4.2]{MR1763974}} %
    \[ %
    (ax)(ya)=a((xy)a), \quad %
    a(x(ay))=((ax)a)y, \quad %
    x(a(ya))=((xa)y)a.  %
    \] %
  \item\label{artin} %
    {\upshape Artin's Theorem (see~\cite[Prop.\,1.5.2]{MR1763974},
      \cite[Thm.\,3.1, p.\,29]{MR0210757})}: 
    For any two elements $x,y\in C$, the subalgebra generated by~$x$
    and~$y$ in~$C$ is associative.
  \item\label{inverse}%
    An element $a\in C$ is invertible if, and only if, its norm is not
    zero; we have $a^{-1}= \N(a)^{-1}\,\gal{a}$ in that case. %
    Thus a non-split composition algebra is a division algebra. %
    Note that Artin's Theorem then implies
    $a^{-1}(ax)=x=a(a^{-1}x)=(xa)a^{-1}=(xa^{-1})a$, for each
    $x\in C$. %
  \item%
    Each element $a\in C$ is a root of a polynomial of degree~$2$
    over~$F$, namely, the polynomial
    $X^2-(a+\gal{a})X+\N(a) \in F[X]$. %
    We call\/ $\tr(a) \coloneqq \tr[C](a)\coloneqq a+\gal{a}$ the
    \emph{trace} of\/~$a$ in~$C$. %
  \item\label{autSemilinear}%
    {\upshape\cite[1.3]{MR3019290}} %
    Every $F$-semilinear automorphism of\/~$C$ commutes with the
    standard involution. %
    If\/ $\dim_FC\ge4$ then every\/ $\ZZ$-linear automorphism of\/~$C$
    is $F$-semilinear.  Consequently, every\/ $\ZZ$-linear
    automorphism of such a\/~$C$ is a semi-similitude of the norm
    form, and every $F$-linear automorphism is an orthogonal map.
    
    We write $\Aut C$ for the group of all $\ZZ$-linear automorphisms
    (i.e., all bijections that are additive and multiplicative), %
    and $\Aut[F]{C}$ for the group of all $F$-linear automorphisms. %
  \item\label{doubling}%
    {\upshape\cite[1.5.3]{MR1763974}} %
    If\/ $D$ is a subalgebra of\/~$C$ with $\dim_FC=2\dim_FD$ and
    $D^\perp\cap D=\{0\}$ then $D^\perp=Dw$ holds for each
    $w\in D^\perp$ with $\N(w)\ne0$, and the multiplication in
    $C=D\oplus D^\perp$ is given by
    $(x+yw)(u+vw) = (xu-\N(w)\gal{v}y)+(vx+y\gal{u})w$. %
  \item\label{coveryByQuaternions}%
    {\upshape\cite[1.6.4]{MR1763974}} %
    Every element of an octonion algebra is contained in a quaternion
    subalgebra. 
    In fact, every element of a split octonion algebra is contained in
    a \emph{split} quaternion algebra, see~\upshape{\cite[5.5]{MR2746044}}.
  \end{enumerate}
\end{nthm}

We use the most comprehensive notion of ``subalgebra'' here:

\begin{defi}
  A \emph{subalgebra} in a given (not necessarily unary, not necessarily
  associative) algebra~$B$ over~$F$ is a vector subspace~$A$
  (over~$F$) which is closed under multiplication.

  In particular, even if~$B$ contains a neutral element of
  multiplication, we do \emph{not} require it to lie in the
  subalgebra. %
  A subalgebra~$A$ may well contain a neutral element for its
  multiplication; this will then be an idempotent centralized by~$A$,
  but need not be a unit element in~$B$ (in general, it will not even
  be centralized by~$B$, and need not be invertible).

  We remark that in~\cite[1.2.1]{MR1763974} much stronger properties
  are required for a \emph{composition subalgebra}; that is a
  subalgebra (in the sense defined above) that contains~$1$, and such
  that the restriction of the norm form has non-degenerate polar form.

  Note also that only subalgebras containing~$1$ are considered
  in~\cite{MR2726555}.%
\end{defi}

\begin{lemm}\label{closedUnderKappa}
  A vector subspace~$V$ of a composition algebra\/~$C$ is closed under
  the standard involution~$\kappa$ if, and only if, it either
  contains~$1$ or is contained in~$1^\perp$.  %
\end{lemm}
\begin{proof}
  The assertion is obvious if $V\subseteq 1^\perp$. %
  If $1\in V$ then $\gal{v} = \tr(v)-v$ lies in~$V$, for each
  $v\in V$. %
  Conversely, if $v\in V\smallsetminus 1^\perp$ then $0 \ne \tr(v) \in
  V\cap F$ yields $1\in V$.   
\end{proof}

Note that a subalgebra will, in general, not be a composition algebra
since the polar form of the restriction of the norm form~$\N$ to the
subalgebra may be degenerate. %

\begin{ndef}[Quaternion algebras]
  In the context of abstract algebras over~$F$, a \emph{split
    quaternion algebra} is an algebra~$H$ over~$F$ which is isomorphic
  to the matrix algebra~$F^{2\times2}$. A \emph{quaternion field} can
  be defined as a skewfield (i.e., an associative division
  algebra) with center~$F$ and dimension~$4$ over~$F$, we call this a
  \emph{non-split quaternion algebra} over~$F$.

  Both classes of quaternion algebras are completely contained in the
  class of composition algebras; their members are exactly the
  composition algebras of dimension~$4$. The required (multiplicative)
  norm is obtained as the determinant in the split case. For the
  non-split case, one uses a two-dimensional subalgebra, and the fact
  that the tensor product with that subalgebra yields a split
  quaternion algebra. See also~\cite{MR1681303}.

  A \emph{quaternion subalgebra} of an octonion algebra over~$F$ is a
  subalgebra~$H$ that is isomorphic (as an abstract algebra) to some
  quaternion algebra over~$F$. 
\end{ndef}

For a quaternion subalgebra~$H$ of an octonion algebra~$\OO$, we do
not require that the restriction of the norm~$\N[\OO]$ to~$H$
satisfies any non-degeneracy conditions, and we do not require
\emph{a~priori} that~$H$ contains the neutral element~$1$
of~$\OO$. However, it turns out that the quaternion algebra is
embedded as nice as can be, see~\ref{SkolemNoether} below.

\begin{lemm}\label{anyTwoInQuat}
  Let\/~$\OO$ be a \emph{non-split} octonion algebra over a
  field\/ $F$ with $\Char{F}\ne2$. Then  any two elements of\/~$\OO$ are
  contained in a quaternion subalgebra of\/~$\OO$. %
\end{lemm}
\begin{proof}
  Consider $a,b\in \OO$. %
  If $a\in F+Fb$ (in particular, if $a\in F$) then the assertion is
  obvious from the result that every single element~$b$ of~$\OO$ is
  contained in a quaternion algebra (this even holds if~$\OO$ is
  split, or if $\Char{F}=2$, see~\cite[1.6.4]{MR1763974}). %
  So assume that $a\notin F+Fb$. Then $A \coloneqq F+Fa$ is a
  quadratic extension field of~$F$, and $\OO$ is a vector space
  over~$A$. So~$Ab$ has trivial intersection with~$A$, and
  $H \coloneqq A+Ab = F+Fa+Fb+Fab$ is a vector space of dimension~$4$
  over~$F$. Clearly, that vector space coincides with the
  (associative) subalgebra generated by $\{a,b\}$ in~$\OO$. The
  restriction $\N|_H$ of the norm form~$\N[\OO]$ is non-singular, and
  then non-degenerate because $\Char{F}\ne2$. So~$H$ is a (non-split)
  composition algebra (of dimension~$4$), and thus a quaternion field.
\end{proof}

\begin{rema}
  If\/ $\OO$ is a non-split octonion algebra of characteristic~$2$
  then~$\OO$ contains commutative (and associative) subalgebras of
  dimension~$4$ (obtained as fixed point sets of involutory
  automorphisms, see~\cite[4.5]{MR3100925}). In fact, such a
  subalgebra forms a totally inseparable field extension of degree~$4$
  and exponent~$1$ over~$F$. Two elements generating such a subalgebra
  do not lie in any common quaternion field, obviously,
  so~\ref{anyTwoInQuat} does not extend to the characteristic two
  case.
\end{rema}

\begin{rema}
  In a split octonion algebra, it is no longer true that any two
  elements lie in a quaternion subalgebra. E.g., if $a,b$ are two
  nilpotent elements such that $ab\ne0$ then the algebra $A_{a,b}$
  generated by them is a Heisenberg algebra (see~\ref{Heisenberg}
  below). Such a Heisenberg algebra contains two-dimensional
  subalgebras with trivial multiplication (such as $Fa+Fab$), and is
  therefore not contained in any quaternion algebra.
  So~\ref{anyTwoInQuat} does not extend to the split case.
\end{rema}

\section{Examples of subalgebras}

The present section collects various examples that will serve as
representatives for orbits of subalgebras under $\Aut[F]\OO$. %
See~\ref{allTotallySingularSubalgebras}, \ref{dimGreaterFour}
and~\ref{unaryDimAtMostFour} below. 
  
\begin{lemm}\label{rightIdealConstruction}
  Let~$\HH$ be a split quaternion algebra, and construct the (split)
  octonion algebra $\OO \coloneqq \HH+\HH w$ with multiplication %
  \( %
  (a+xw)(b+yw) = ab+\gal{y}x + (ya+x\gal{b})w \,. %
  \) %
  Let\/~$A$ be a subalgebra of\/~$\HH$, and let\/~$R$ be a vector
  subspace of\/~$\HH$ such that $RA\subseteq R$ and\/
  $\gal{R}R \coloneqq \smallset{\gal{y}x}{x,y\in R} \subseteq A$.
  Then \( A+Rw \) is a subalgebra of\/~$\OO$.

  These conditions are satisfied if\/~$A$ is a subalgebra with
  $\gal{A}=A$ and\/~$R$ is a right ideal in~$A$.
\end{lemm}
\begin{proof}
  For $(a,x),(b,y)\in A\times R$, we use $\gal{R}R \subseteq A$ and
  $RA\subseteq R$ to see that the product %
  $({a+xw})({b+yw}) = ab+\gal{y}x + (ya+x\gal{b})w$ %
  is of the form $c+zw$ with $(c,z)\in A\times R$. %
  If $R\subseteq A=\gal{A}$ then $\gal{R}R \subseteq A$. %
\end{proof}

A vector subspace~$A$ of~$\OO$ satisfies $\gal{A}=A$ precisely if
either $1\in A$ or $A\subseteq 1^\perp$, see~\ref{closedUnderKappa}.

\begin{exas}\label{Exasleftidealconstruction}
  Recall that each split quaternion algebra over~$F$ is isomorphic
  to~$F^{2\times2}$. %
  We abbreviate $n_0 = \left(
    \begin{smallmatrix}
      0 & 1 \\
      0 & 0
    \end{smallmatrix}
  \right)$ and
  $p_0 = \left(
    \begin{smallmatrix}
      1 & 0 \\
      0 & 0
    \end{smallmatrix}
    \right)$. 
  \begin{enumerate}
  \item Apart from $F^{2\times2}$ and~$\{0\}$, each right ideal
    of~$F^{2\times2}$ has dimension~$2$, and is of the form $R_v
    \coloneqq \set{X\in F^{2\times2}}{vX=0}$ for a row vector
    $v=(v_1,v_2) \ne(0,0)$. %

    For any such~$R_v$, the subalgebra $F^{2\times2}+R_vw$ has
    dimension~$6$.  We will prove that every subalgebra of
    dimension~$6$ is in the orbit of this subalgebra under
    $\Aut[F]\OO$, see~\ref{dimGreaterFour}. %
    Note that $F^{2\times2}+R_{(0,1)}w = \OO(n_0w)+\OO(p_0w)$. %
  \item Every $3$-dimensional subalgebra of~$F^{2\times2}$ is a
    conjugate of the subalgebra $U \coloneqq \set{\left(
        \begin{smallmatrix}
          r & t \\
          0 & s
        \end{smallmatrix}\right)}{r,s,t\in F}$ of upper triangular
    matrices. %
    The right ideals (apart from $\{0\}$ and~$U$) in~$U$ are
    $L \coloneqq \set{\left(
      \begin{smallmatrix}
        r & t \\
        0 & 0 
      \end{smallmatrix}\right)}{r,t\in F}$, %
    $\kappa(L) = \set{\left(
      \begin{smallmatrix}
        0 & t \\
        0 & s 
      \end{smallmatrix}\right)}{t,s\in F}$, %
  and %
  $M_y \coloneqq \set{\left(
      \begin{smallmatrix}
        0 & y_1s \\
        0 & y_2s
      \end{smallmatrix}\right)}{s\in F}$ %
  for $y=\binom{y_1}{y_2} \in F^{2\times1}$. %
  By~\ref{rightIdealConstruction}, this gives examples of
  subalgebras of dimension~$5$. %
  We will prove that every subalgebra of dimension~$5$ is in the orbit
  of $U+Lw$ under $\Aut[F]\OO$, see~\ref{dimGreaterFour}. %
  Note that $U+Lw = n_0\OO+\OO n_0$.
  \item Let $A$ be any subalgebra of\/~$\HH$, and let
    $R\coloneqq\set{\left(
        \begin{smallmatrix}
          r & t \\
          0 & 0
        \end{smallmatrix}\right)}{r,t\in F}$.
    Then $\gal{R}R = \{0\}$, and $A+Rw$ is a subalgebra
    by~\ref{rightIdealConstruction}.  In this way, one finds examples %
    of subalgebras of dimension $\dim{A}+2 \in \{3,4,5\}$.
    See~\ref{unaryDimAtMostFour}\ref{QtwoDim}. %
  \item The set $A \coloneqq \set{\left(
      \begin{smallmatrix}
        r & t \\
        0 & 0 
      \end{smallmatrix}\right)}{r,t\in F}$ %
  is a subalgebra of\/~$\HH$ with $\gal{A}\ne A$. Using
   $R \coloneqq F\left(
      \begin{smallmatrix}
        0 & 1 \\
        0 & 0 
      \end{smallmatrix}\right)$ %
    and~\ref{rightIdealConstruction} we obtain a three-dimensional
    subalgebra $A+Rw$ in~$\OO$. %
    \\%
    See~\ref{allTotallySingularSubalgebras}\ref{mOOcapOOn}. %
    Note that $A+Rw = n_0\OO\cap\OO(n_0w)$. %
\end{enumerate}
\end{exas}

\begin{exam}\label{fourdimNotAssoc}
  In $\OO = F^{2\times2}+F^{2\times2}w$ with $w^2=1$, consider $n_0
  \coloneqq \left(
    \begin{smallmatrix}
      0 & 1 \\
      0 & 0 
    \end{smallmatrix}\right)$ and %
  $p_0 \coloneqq \left(
    \begin{smallmatrix}
      1 & 0 \\
      0 & 0 
    \end{smallmatrix}\right)$. %
  The $4$-dimensional subalgebra %
  \( %
  n_0\OO = \set{\left(
      \begin{smallmatrix}
        r & s \\
        0 & 0
      \end{smallmatrix}
    \right) + \left(
      \begin{smallmatrix}
        0 & t \\
        0 & u
      \end{smallmatrix}
    \right)w}%
  {r,s,t,u\in F} %
  \) %
  (obtained as in~\ref{rightIdealConstruction} with $A = Fp_0+Fn_0$
  and\/ $R = (Fn_0+F\gal{p_0})w$) %
  is not associative. In fact, we have %
  $p_0\bigl((\gal{p_0}w)(n_0w)\bigr) = p_0(\gal{n_0}\,\gal{p_0}) =
  -n_0$ %
  while
  $\bigl(p_0(\gal{p_0}w)\bigr)(n_0w) =
  \bigl((\gal{p_0}p_0)w\bigr)(n_0w) = 0$.
\end{exam}

In the present paper, we will prove that every $3$-dimensional
subalgebra of any octonion algebra is associative
(see~\ref{allTotallySingularSubalgebras} for algebras that do not
contain~$1$; %
in the case where $1\in A$ and $\dim{A}=3$ choose a basis $1,a,b$
for~$A$, note that $a,b$ generate an associative subalgebra $B$
by~\ref{propertiesCA}\ref{artin}, and that $A = F+B$ is associative),
and that every subalgebra of dimension greater than~$4$ is
non-associative (see~\ref{allTotallySingularSubalgebras}
and~\ref{dimGreaterFour}). %

\begin{lemm}\label{1inSub}
  A subalgebra of\/~$\OO$ contains~$1$ if, and only if, it contains an
  invertible element.
\end{lemm}
\begin{proof}
  Assume that~$A$ contains an invertible element~$a$. Then
  $\gal{a}a = (a+\gal{a})a - a^2$ is a non-trivial element of
  $F\cap A$, and $1\in A$ follows.  %
\end{proof}

\begin{coro}
  Every subalgebra of dimension greater than four contains~$1$. %
  \qed
\end{coro}

For any $x,y\in\OO$, the subalgebra generated by $\{1,x,y\}$ in~$\OO$
is associative (by Artin's Theorem,
see~\ref{propertiesCA}\ref{artin}).  We investigate the borderline
between associative and non-associative subalgebras in more detail:

\begin{lemm}
  Let\/~$A$ be a subalgebra of an octonion algebra~$\OO$
  over~$F$. If\/~$A$ is generated by two elements then $\dim{A}\le4$;
  in fact, we have $A\subseteq F+Fx+Fy+Fxy$ if~$A$ is generated by
  $\{x,y\}$.
\end{lemm}
\begin{proof}
  We abbreviate $\xi \coloneqq x+\gal{x}$ and
  $\eta \coloneqq y+\gal{y}$. Then $\xi$ and~$\eta$ lie in~$F$, and
  $yx = (yx+\gal{yx})-\gal{yx}$ lies in $F+F(\gal{yx})$. Thus it
  suffices to note that
  $\gal{yx} = \gal{x}\,\gal{y} = (\xi-x)(\eta-y) = \xi\eta - \xi y -
  \eta x + xy$ lies in $F+Fx+Fy+Fxy$.
\end{proof}

\begin{ndef}[Heisenberg subalgebras] %
  \label{Heisenberg}%
  Take a nilpotent element $a\ne0$ (i.e., any non-zero element~$a$ of
  norm and trace $0$), and pick a nilpotent element
  $b\in a^\perp\smallsetminus ({F+Fa})$. %
  Then $\N(ab)=0$, and $(ab|1) = (b|\gal{a}) = (b|-a) = 0$
  (see~\ref{propertiesCA}\ref{adjoint}) gives $ba=-ab$. %
  Moreover, computation in the (associative) subalgebra generated
  by~$a$ and~$b$ yields $a(ab) = -\N(a)b = 0$,
  $(ab)a = -(ba)a{{} = 0}$. %
  Analogously, we find $b(ab) = 0 = (ab)b$ and $(ab)(ab) = 0$.  So the
  linear span of $\{a,b,ab\}$ is closed under multiplication, and
  forms a subalgebra $A_{a,b}$.

  There are two cases:
  \begin{enumerate}
  \item\label{twoDimNull} %
    If $ab = 0$ then $ba = \gal{ab} = 0$, and $A_{a,b} = Fa+Fb$ is a
    two-dimensional subalgebra such that the restriction
    $\N|_{A_{a,b}}$ is trivial. %
    We then conclude that the multiplication on~$A_{a,b}$ is trivial
    (i.e., every product of two elements in~$A_{a,b}$ equals~$0$). %
  \item\label{heisenbergIndeed}%
    If $ab\ne0$ then $ab\notin Fa+Fb$ because $ab=ra+sb$ with
    $r,s\in F$ implies $0 = a(ab) = ra^2+s(ab) = s(ab)$ and
    $0 = (ab)b = r(ab)+sb^2 = r(ab)$, leading to the contradiction
    $s=0=r$. In this case, the algebra $A_{a,b}$ has dimension~$3$,
    and is isomorphic to the Heisenberg algebra over~$F$. %

    See~\ref{dim3inPu} below for a characterization of Heisenberg
    algebras in~$\OO$, and a proof that the Heisenberg subalgebras
    form a single orbit under $\Aut[F]\OO$.
  \end{enumerate}
\end{ndef}

\begin{exas}
  In $\OO = F^{2\times2}+F^{2\times2}w$ with $w^2=1$,
  the element $p_0 = \left(
    \begin{smallmatrix}
      1 & 0 \\
      0 & 0 
    \end{smallmatrix}\right)$
  is an idempotent (i.e., $p_0^2=p_0$), 
  the elements $n_0 = \left(
    \begin{smallmatrix}
      0 & 1 \\
      0 & 0
    \end{smallmatrix}\right)$ and $n_0w$ are nilpotent,  %
  and so is $p_0w$. 

  Computing $n_0(p_0w) = n_0w$ and $n_0(n_0w)=0$, we obtain that
  $Fn_0+Fn_0w = A_{n_0,n_0w}$ is a two-dimensional algebra as
  in~\ref{Heisenberg}\ref{twoDimNull}, and that the $3$-dimensional
  algebra $A_{n_0,p_0w} = {Fn_0+Fp_0w+Fn_0w}$ is an example of a
  Heisenberg algebra, as in~\ref{Heisenberg}\ref{heisenbergIndeed}.
  Those algebras are obtained by the method given
  in~\ref{rightIdealConstruction}, using $A = Fn_0$ and
  $R\in\{Fn_0,Fn_0+Fp_0\}$.
\end{exas}

\section{Centralizers}

\begin{lemm}\label{centralizers}
  Let\/ $\OO$ be an octonion algebra over~$F$, and consider\/
  $v\in\OO\smallsetminus F$.
  \begin{enumerate}
  \item If\/ $\gal{v}=v$ then $\Char{F}=2$, and the centralizer
    $\C[\OO]{v}$ of\/~$v$ in~$\OO$ has dimension~$6$. In this case,
    the centralizer is not closed under multiplication. %
    That centralizer consists of pure elements; in fact we have
    $\C[\OO]{v} = \{1,v\}^\perp$. %
  \item If\/ $\gal{v}\ne v$ (in particular, if\/ $\Char{F}\ne2$) then
    $\dim{\C[\OO]{v}} = 2$ if\/ $\N(v-\gal{v})\ne0$, and\/
    $\dim{\C[\OO]{v}} = 4$ if\/ $\N(v-\gal{v})=0$.
  \item If\/ $\dim{\C[\OO]{v}} = 4$ then $\Char{F}\ne2$, there exists
    $n\in F+Fv$ with $n^2=0\ne n$ (so~$\OO$ is split), and\/
    $\C[\OO]{v}$ is a subalgebra. %
  \end{enumerate}
\end{lemm}
\begin{proof}
  First of all, we choose a quaternion subalgebra~$H$ of~$\OO$
  containing~$v$ (see~\ref{propertiesCA}\ref{coveryByQuaternions}),
  and write $\OO = H+Hw$ as a double of~$H$, with multiplication
  \[
    (a+xw)(b+yw) = ab-\N(w)\gal{y}x + (ya+x\gal{b})w \,.
  \]
  Now $a+xw$ centralizes~$v$ precisely if $av=va$ and
  $x(v-\gal{v})=0$. In any case, this means that $a$ belongs to the
  two-dimensional centralizer $F+Fv$ of~$v$ in~$H$. 

  If $\gal{v}=v$ then $v\notin F$ implies $\Char{F}=2$. In this case,
  $a\in F+Fv$ but we have no restriction on~$x$, and the centralizer has
  dimension~$6$. However, the product
  $ab-\N(w)\gal{y}x + (ya+x\gal{b})w$ of elements $(a+xw)$ and
  $(b+yw)$ of the centralizer belongs to that centralizer only if
  $\gal{y}x$ lies in $F+Fv$. %
  
  Now assume $\gal{v}\ne v$. If $\N(v-\gal{v})\ne0$ then
  $a+xw\in\C[\OO]{v}$ implies $x=0$, and $\C[\OO]{v} = F+Fv$ has
  dimension~$2$. %
  If $\N(v-\gal{v})=0$ then~$H$ is a split quaternion algebra
  (isomorphic to $F^{2\times2}$), and the annihilator
  $\set{x\in H}{x(v-\gal{v})=0}$ has dimension~$2$. Then
  $\C[\OO]{v} = \set{a+xw}{a\in F+Fv, x(v-\gal{v})=0 }$ has
  dimension~$4$.  That centralizer is a subalgebra
  by~\ref{rightIdealConstruction}.
\end{proof}

\begin{exam}\label{Cn0}
  Consider $n_0 = \left(
    \begin{smallmatrix}
      0 & 1 \\
      0 & 0 
    \end{smallmatrix}\right)
  $ and $p_0 = \left(
    \begin{smallmatrix}
      1 & 0 \\
      0 & 0 
    \end{smallmatrix}\right)
  $ in $\OO = F^{2\times2}+F^{2\times2}w$, with $w^2=1$. %
  If $\Char{F}\ne2$ then %
  $\C[\OO]{n_0} = F+Fn_0+Fn_0w+F\gal{p_0}w$.  This centralizer forms
  an associative subalgebra of~$\OO$, isomorphic to
  \[
    \set{\left(
        \begin{smallmatrix}
          \alpha & \beta & \gamma & \delta \\
          0 & \alpha & 0 & 0 \\
          0 & -\delta & \alpha & 0 \\
          0 & \gamma & 0 & \alpha
        \end{smallmatrix}
      \right)}{\alpha,\beta,\gamma,\delta\in F}.
  \]
  Note that $\C[\OO]{n_0} = F+(n_0\OO\cap\OO n_0)$ holds in this case.

  If $\Char{F}=2$ then $n_0-\gal{n_0} = 0$, and
  $\C[\OO]{n_0} = F+Fn_0+F^{2\times2}w$ has dimension~$6$. %
\end{exam}

\begin{coro}
  \begin{enumerate}
  \item If\/ $\Char{F}=2$ then the centralizer $\C[\OO]v$ has
    dimension~$6$ if\/ $v\in 1^\perp\smallsetminus F$, and it has
    dimension~$2$ if\/ $v\in\OO\smallsetminus 1^\perp$.
  \item %
    In any case, the centralizer $\C[\OO]v$ of\/ $v\in\OO\smallsetminus F$
    has dimension~$2$ if, and only if, the subalgebra $F+Fv$ is a
    composition algebra. %
    Then $F+Fv = \C[\OO]v$. 
  \item If\/ $\Char{F}\ne2$ then $\dim\C[\OO]v = 4$ holds if\/
    $v\in\OO\smallsetminus F$ and\/
    $0 = \N(v-\gal{v}) = -(v-\gal{v})^2$. %
    In that case, we note that the centralizers of $v$ and of
    $v-\gal{v} = 2v-\tr(v)$ then coincide, and that $v-\gal{v}$ is
    nilpotent. %
    As every non-trivial nilpotent element of~$\OO$ lies in the same
    orbit under $\Aut[F]\OO$ (see~\ref{normAndTraceCharacterizeOrbits}
    below), all these centralizers are in the orbit of the one
    described in~\ref{Cn0}.  %
    \qed
  \end{enumerate}
\end{coro}

\begin{rema}\label{detailsCn}
  If $\C[\OO]v$ has dimension~$4$ then there exists
  $n\in({1^\perp\cap ({F+Fv})})\smallsetminus F$ such that $v\in F+n$,
  and $\C[\OO]v = F+(n\OO\cap\OO n) = \C[\OO]n$. See~\ref{Cn0} for an
  explicit example, and also~\ref{threedimPureSubalgebra} below. %
  This is not a composition algebra: the restriction of the norm form
  has Witt index~$3$, and is degenerate. The non-invertible elements
  form a hyperplane, which is an ideal. %
  So~$\C[\OO]v$ is a local algebra, in this case.
\end{rema}

\begin{coro}\label{noQuaternion}
  If $v\notin F$ then no quaternion subalgebra is contained in~$\C[\OO]{v}$. %
  \qed
\end{coro}

A vector subspace~$V$ in a composition algebra~$C$ is called %
\emph{totally singular} if $\N[C](V) = \{0\}$.  The maximal totally
singular subspaces in split octonion algebras are well known:

\begin{theo}\label{vanderBljSpringer}%
Let\/ $\OO$ be a split octonion algebra.
\begin{enumerate}
\item\label{kerLambdaA} For any $a\in\OO$, we have $a\OO = \ker\lambda_{\gal{a}}$
  (see~{\upshape\cite[Lemma\,1]{MR0123622}}).
\item\label{maxTotSing}%
  If $\N(a)=0$ but $a\ne0$ then $a\OO$ and $\OO a$ are maximal
  totally singular subspaces, and every maximal totally singular
  subspace is of this form
  (see~{\upshape\cite[Theorem\,3]{MR0123622}}).
\item\label{intersectMaxTotSing}%
  For $a,b\in\OO\smallsetminus\{0\}$ with $\N(a) = 0 = \N(b)$ we have
  (see~{\upshape\cite[Theorem\,4, Theorem\,5]{MR0123622}})
  \begin{enumerate}
  \item $a\OO = b\OO \iff Fa=Fb$, %
  \item $\dim(a\OO\cap b\OO)=2 \iff (a|b)=0 \land \dim(Fa+Fb)=2$; \\%
    then $a\OO\cap b\OO = a(\gal{b}\OO) = b(\gal{a}\OO)$, %
  \item $a\OO\cap b\OO = \{0\} \iff (a|b)\ne0$,
  \item $a\OO \cap \OO b = Fab\iff ab\ne0$, %
  \item $\dim(a\OO \cap \OO b) = 3 \iff ab=0$; \\%
    then $a\OO \cap \OO b = \set{ay}{y\in b^\perp} = \set{xb}{x\in a^\perp}$.
  \end{enumerate}
\end{enumerate}
As the standard involution interchanges $\set{a\OO}{a\in\OO}$ with
$\set{\OO a}{a\in\OO}$, the list above covers all possibilities for
the relative position (i.e., the size of the intersection) of maximal
totally singular subspaces in~$\OO$. 
    \qed
\end{theo}

Consider a $3$-dimensional subspace $T$ such that $\N|_T=0$.  Since
the intersection $a\OO\cap a'\OO$ has dimension at most~$2$ if
$a\OO\ne a'\OO$, there is precisely one subspace $a\OO$ and precisely
one subspace $\OO b$ such that $T\subseteq a\OO$ and
$T\subseteq \OO b$ (i.e., such that $T=a\OO\cap\OO b$).

\begin{exas}\label{threedimPureSubalgebra}
  Consider %
  $p_0 = \left(
    \begin{smallmatrix}
      1 & 0 \\
      0 & 0
    \end{smallmatrix}\right)$, %
  and %
  $n_0 = \left(
    \begin{smallmatrix}
      0 & 1 \\
      0 & 0
    \end{smallmatrix}\right)$ %
  in the split octonion algebra
  $\OO \coloneqq F^{2\times2} + F^{2\times2}w$, where $w^2=1$. %
  \begin{enumerate}
  \item\label{pownow}%
    The set $(p_0w)\OO \cap \OO(n_0w) = Fp_0+Fp_0w+Fn_0w$ is a
    subalgebra of dimension~$3$, and every element is obtained as a
    product in the subalgebra.  This subalgebra is isomorphic to the
    associative algebra
    \[
      \set{\left(
          \begin{matrix}
            \alpha&0&0&\beta\\
            0&\alpha&0&0\\
            0&\gamma&0&0\\
            0&0&0&0\\
          \end{matrix}
        \right)}{\alpha,\beta,\gamma\in F} \,;
    \]
    it is also obtained by the construction
    in~\ref{rightIdealConstruction} with $A=Fp_0$ and $R=Fp_0+Fn_0$.
  \item The algebra
    $n_0\OO\cap \OO n_0 = \set{\alpha n_0+(\beta
      n_0+\gamma\gal{p_0})w}{\alpha,\beta,\gamma\in F}$ %
    has multiplication
  \[
    (\alpha p_0+(\beta p_0+\gamma n_0)w) %
    (\xi p_0+(\eta p_0 + \zeta n_0)w) %
    = %
    (\beta\zeta - \gamma\eta)n_0 ;
  \]
  this is a Heisenberg algebra, as
  in~\ref{Heisenberg}\ref{heisenbergIndeed}). %
  It is neither isomorphic nor anti-isomorphic to
  $(p_0w)\OO \cap \OO(n_0w)$. %
  This Heisenberg algebra is also obtained by the construction
  in~\ref{rightIdealConstruction}, with $A=Fn_0$ and
  $R=Fn_0+F\gal{p_0}$.
  \end{enumerate}
\end{exas}

\section{Automorphisms of (split) composition algebras}

The following is a variant of the well known Skolem-Noether Theorem
(e.g., see~\cite[4.6, p.\,222]{MR1009787} or~\cite[5.1.7,
p.\,183]{MR1953987}). %
Note, however, that we deal with subalgebras of a non-associative
algebra, and do not require a priori that the subalgebras
contain~$1$. 

\begin{lemm}\label{SkolemNoether}
  Let\/ $\OO$ and\/~$\OO'$ be octonion algebras over~$F$, and let\/
  $H<\OO$ and\/ $H'<\OO'$ be quaternion subalgebras. %
  Then both~$H$ and\/~$H'$ contain the neutral element~$1$
  of~$\OO$ and of\/~$\OO'$, respectively. Moreover, we have: 
  \begin{enumerate}
  \item\label{wittPropertyHHinOO}%
    If\/ $\OO$ and $\OO'$ are isomorphic then every non-constant
    $F$-linear multiplicative map from~$H$ to~$H'$ extends to an
    isomorphism from~$\OO$ onto~$\OO'$.
  \item%
    If\/ $H$ or~$H'$ is split then every non-constant
    $F$-linear multiplicative map from~$H$ to~$H'$ extends to an
    isomorphism from~$\OO$ onto~$\OO'$.
  \end{enumerate}
\end{lemm}
\begin{proof}
  Let $e$ and $e'$ be the neutral elements of multiplication in~$H$
  and~$H'$, respectively. Then~$H$ is a quaternion subalgebra
  of~$\OO$ centralizing the idempotent~$e$. From~\ref{noQuaternion} we
  infer $e\in F$ and then $e=1$. 

  Now consider a non-constant $F$-linear multiplicative map
  $\beta\colon H \to H'$. As $H$ is a simple algebra, the ideal
  $\ker\beta$ is trivial, and $\beta$ is an isomorphism by dimension
  reasons. %
  The restrictions $\N[\OO]|_H$ and $\N[\smash{\OO'}]|_{\smash{H'}}$
  of the norms are isometric because the quaternion subalgebras are
  isomorphic. By Witt's Cancellation Theorem (\cite[Satz\,4,
  p.\,34]{MR1581519}, \cite[Folgerung, p.\,160]{MR0008069},
  see~\cite[p.\,21, p.\,35]{MR0072144} or~\cite[5.1,
  12.12]{MR1859189}) %
  the restrictions $\N[\OO]|_{H^\perp}$ and
  $\N[\smash{\OO'}]|_{\smash{H'}^\perp}$ of the norms are isometric,
  as well. We pick any element $w\in H^\perp$ with $\N(w)\ne0$, and
  find $w'\in\smash{H'}^\perp$ with
  $\N[\OO](w')=\N[\smash{\OO'}](w)$. The formula
  (see~\ref{propertiesCA}\ref{doubling}) for the multiplication in
  $\OO=H+H w$ and in $\OO' = H'+H'w'$, respectively, immediately
  yields that mapping $a+xw \in H+H w$ to
  $\beta(a)+\beta(x)w' \in H'+H'w'$ is an automorphism of~$\OO$. %

  If $H$ or $H'$ is split then the existence of an isomorphism
  between the quaternion algebras implies that both are split, and so
  are the octonion algebras~$\OO$ and~$\OO'$. So the latter are
  isomorphic, and our first assertion applies.
\end{proof}

\begin{rema}\label{stabilizerH}
  Let\/ $\OO=H+Hw$ be obtained by doubling a quaternion subalgebra~$H$, using
  $w\in H^\perp$. 
  The $F$-linear automorphisms of~$\OO$ that leave~$H$ invariant can
  be given explicitly: they are the maps $\alpha_{s,t}\colon H+Hw\to
  H+Hw\colon a+xw \mapsto sas^{-1}+(txs^{-1})w$, where $s,t$ are
  invertible elements in~$H$ with $\N(s)=\N(t)$. 

  In fact, the restriction of a given automorphism
  $\alpha\in\Aut[F]\OO$ with $\alpha(H)=H$ to~$H$ gives an $F$-linear
  automorphism of~$H$, which is inner by the Skolem-Noether Theorem %
  (e.g., see~\cite[4.6, p.\,222]{MR1009787} or~\cite[5.1.7,
  p.\,183]{MR1953987}). %
  So there exists $s\in H$ with $\N(s)\ne0$ and $\alpha(x)=sxs^{-1}$
  for each $x\in H$. Now $\alpha(w)$ lies in $H^\perp$, and is of the
  form $\alpha(w)=uw$ with $u\in H$. As $\alpha$ preserves the norm,
  we have $\N(u)=1$, and obtain
  $\alpha(xw) = \alpha(y)\,\alpha(w) = (sxs^{-1})(uw) = (usxs^{-1})w$,
  for each $x\in H$. %
  Putting $t\coloneqq us$ gives the description as claimed.
\end{rema}

\begin{coro}\label{quaternionOrbit}
  In any octonion algebra~$\OO$, two quaternion subalgebras are in the
  same orbit under $\Aut[F]\OO$ if, and only if, they are isomorphic
  (as $F$-algebras). %
  \qed
\end{coro}

\begin{lemm}\label{normAndTraceCharacterizeOrbits}%
  Let\/~$\OO$ be an octonion algebra over a field\/~$F$. %
  Then non-central elements of\/~$\OO$ are in the same orbit under
  $\Aut[F]\OO$ if, and only if, they have the same norm and the same
  trace.
\end{lemm}
\begin{proof}
  Clearly elements in the same orbit have the same norm and trace.

  Now consider $a_1,a_2\in\OO\smallsetminus F$ with $\N(a_1)=\N(a_2)$ and
  $\tr(a_1)=\tr(a_2)$. Applying Witt's theorem, we find $b_j\in
  a_j^\perp$ with $\tr(b_1)=\tr(b_2)=1$ and
  $\N(b_1)=\N(b_2)\ne0$. 

  If $\N(a_1)\ne0$ then $H_j \coloneqq F+Fb_j+(F+Fb_j)a_j$ is a
  quaternion algebra, and the $F$-algebras $H_1$ and~$H_2$ are
  isomorphic because their norm forms are isometric
  (see~\cite[1.7.1]{MR1763974}). %
  By~\ref{SkolemNoether}\ref{wittPropertyHHinOO}, any $F$-linear
  isomorphism from~$H_1$ onto~$H_2$ extends to an automorphism
  of~$\OO$. If $\alpha\in\Aut[F]\OO$ is such an extension then the
  image $\alpha(a_1) \in H_2$ has the same norm and trace as~$a_2$,
  and is thus a conjugate of~$a_2$ in~$H_2$. Adapting~$\alpha$ by the
  extension of a suitable (inner) automorphism of~$H_2$, we thus
  obtain an $F$-linear automorphism of~$\OO$ mapping~$a_1$ to~$a_2$.

  If $\N(a_1)=0$ then $a_j$ is contained in a split quaternion
  subalgebra~$H_j$
  (see~\ref{propertiesCA}\ref{coveryByQuaternions}). %
  Since any two split quaternion algebras over~$F$ are isomorphic, we
  can proceed as above.
\end{proof}

\section{Totally singular and totally isotropic subalgebras}
\label{sec:singular}

\enlargethispage{5mm}%
In this section, we study subalgebras consisting of elements of
norm~$0$, i.e., such that the underlying vector space is %
\emph{totally singular with respect to the quadratic form~$\N$}. %
These are just those subalgebras that do not contain~$1$;
see~\ref{1inSub}.  Throughout this section, we write
$\OO = \HH +\HH w$ as the double of the split quaternion algebra
$\HH = F^{2\times2}$ with an element $w$ %
of norm~$-1$ (i.e., with $w^2=1$); then
\[
  (a+xw)(b+yw) = ab+\gal{y}x + (ya+x\gal{b})w
\]%
holds for all $a,b,x,y\in F^{2\times2}$;
see~\ref{propertiesCA}.\ref{doubling}.

\begin{lemm}\label{pureAlgebrasIsotropic}
  If\/ $A$ is a subalgebra of\/~$\OO$ with $A\subseteq 1^\perp$
  then~$A$ is totally isotropic (with respect to the polar form of the
  norm), and $\dim{A}\le4$.  
\end{lemm}
\begin{proof}
  For $a,b\in A$, we compute $(a|b) = (1|\gal{a}b) = (1|ab) = 0$,
  using $ab\in A \subseteq 1^\perp$. 
\end{proof}

Note that a subalgebra contained in~$1^\perp$ need not be totally
\emph{singular} if $\Char{F}=2$; for instance, consider $F+Fn_0$, or a
totally inseparable field of degree~$4$ and exponent~$1$, over a field
$F$ of characteristic two. %

\begin{lemm}\label{subalgebrasInpOO}
  Let\/~$p$ be any element with $\N(p)=0$ but
  $\tr(p)\ne0$. Then $\dim{A}\le2$ holds for every subalgebra
  contained in~$p\OO$ or in~$\OO p$, and every two-dimensional algebra
  in~$p\OO$ (or in~$\OO p$) contains~$p$.
\end{lemm}
\begin{proof}
  Replacing~$p$ by $\tr(p)^{-1}p$ we do not change~$p\OO$. %
  Using~\ref{normAndTraceCharacterizeOrbits}, we may therefore assume
  $p=\left(
    \begin{smallmatrix}
      1 & 0 \\
      0 & 0
    \end{smallmatrix}\right)$.  %
  Then $p\OO = \set{\left(
      \begin{smallmatrix}
        a_1 & a_2 \\
        0 & 0
      \end{smallmatrix}\right) + %
    \left(
      \begin{smallmatrix}
        x_1 & 0 \\
        x_2 & 0 
      \end{smallmatrix}\right)w}{a_1,a_2,x_1,x_2\in F}$. 

  Let~$A$ be a subalgebra in~$p\OO$, and consider the product of two
  arbitrary elements of~$A$. Then
  \[
    \begin{array}{rl}
      \left( %
      \left( %
      \begin{matrix}
        a_1 & a_2 \\
        0 & 0
      \end{matrix}
            \right) + %
            \left( %
            \begin{matrix}
              x_1 & 0 \\
              x_2 & 0
            \end{matrix}
                    \right)w %
                    \right) %
                    \left( %
                    \left( %
                    \begin{matrix}
                      b_1 & b_2 \\
                      0 & 0
                    \end{matrix}
                          \right) + %
                          \left( %
                          \begin{matrix}
                            y_1 & 0 \\
                            y_2 & 0
                          \end{matrix}
                                  \right)w %
                                  \right) %
      \\[2ex] =  %
      \left( %
      \begin{matrix}
        a_1b_1 & a_1b_2 \\
        y_1x_2-y_2x_1 & 0
      \end{matrix}
                        \right) + %
                        \left( %
                        \begin{matrix}
                          y_1a_1 & y_1a_2-x_1b_2 \\
                          y_2a_1 & y_2a_2-x_2b_2
                        \end{matrix}
                                   \right) w %
          & \in p\OO %
    \end{array}
  \]
  yields the conditions \( %
  y_1x_2-y_2x_1 = 0 %
  \), %
  \( %
  y_1a_2-x_1b_2 = 0 %
  \), %
  and \( %
  y_2a_2-x_2b_2 = 0 %
  \). %
  This means that the matrix
  \[
    \left( %
      \begin{matrix}
        a_2 & x_1 & x_2 \\
        b_2 & y_1 & y_2 \\
      \end{matrix}
    \right)
  \]
  has rank at most one. So the quotient space~$(A+Fp)/Fp$ has
  dimension at most one.

  The assertions about $\OO p$ are obtained by applying
  the anti-automorphism~$\kappa$. 
\end{proof}

The maximal totally singular subspaces in~$\OO$ (with respect to the
quadratic form~$\N = \N[\OO]$) are just the sets $a\OO$ and~$\OO a$, with
$a\in\OO\smallsetminus\{0\}$ such that $\N(a)=0$, and one knows
$a\OO = \ker(\lambda_{\gal{a}})$ and $\OO a = \ker(\rho_{\gal{a}})$
(see~\ref{vanderBljSpringer}\ref{kerLambdaA},\,\ref{maxTotSing}).

We characterize the subalgebras among these subspaces: 

\begin{lemm}\label{maximalIsotropicSubalgebras}
  Let\/ $a\in\OO\smallsetminus\{0\}$ with $\N(a)=0$. %
  Then the following are equivalent.
  \begin{enumerate}
    \begin{multicols}{2}
    \item $a\OO$ is a subalgebra.
    \item $\OO a$ is a subalgebra.
    \end{multicols}
    \begin{multicols}{2}
    \item $\tr(a)=0$.
    \item $a^2 = 0 \ne a$. %
    \end{multicols}
  \item After identification of\/~$\OO$ with a double
    of\/~$F^{2\times2}$, we have~$a$ in the $\Aut[F]\OO$-orbit
    of\/~$n_0$.
  \end{enumerate}
  \noindent%
  If one (and then any) of these conditions is satisfied then the
  algebras $a\OO$ and $\OO a$ are anti-isomorphic (via~$\kappa$) but
  not isomorphic.

  These subalgebras have dimension~$4$ and contain no invertible
  elements at all. None of these subalgebras is associative
  (see~\ref{fourdimNotAssoc}), and none of them is contained
  in~$1^\perp$.
\end{lemm}
\begin{proof}
  Since the restriction $\N|_{1^\perp}$ is a non-degenerate quadratic
  form, each totally singular subspace of~$1^\perp$ has dimension at
  most~$3$. Therefore, we have $a\OO\not\subseteq 1^\perp$ for each
  $a\ne0$. %
  If $\tr(a)\ne0$ we apply~\ref{subalgebrasInpOO} to see
  that neither~$a\OO$ nor~$\OO a$ is a subalgebra. 

  So assume~$\tr(a)=0$ now. Then
  $\ker(\lambda_a) = \ker(\lambda_{\gal{a}}) = a\OO$. Using
  $a^2=-\gal{a}a = 0$ and one of Moufang's identities
  (see~\ref{propertiesCA}\ref{MoufangId}) we verify that
  $(a\OO)(a\OO)$ is contained in $\ker(\lambda_a)$ by computing
  $a\left(\strut(ax)(ay)\right) = \left((a\strut(ax))a\right)y = %
  \left((a^2x)a\right)y = 0$. %
  So~$a\OO$ is a subalgebra. %
  Applying the anti-automorphism~$\kappa$ we translate this
  observation to~$\OO a$. Moreover, we see that $a\OO$ and $\OO a$ are
  anti-isomorphic if $\tr(a) = 0 = \N(a)$.

  It remains to show that the algebras $a\OO$ and $\OO a$ are not
  isomorphic. %
  To this end, we note that $a$ is a non-trivial element of~$a\OO$
  such that $a(a\OO) = \{0\}$ while no element~$b$ of
  $a\OO\smallsetminus\{0\}$ satisfies $(a\OO)b = \{0\}$; in fact, the
  kernel $\OO\gal{b}$ of $\rho_b$ never coincides with~$a\OO$
  (see~\ref{vanderBljSpringer}\ref{intersectMaxTotSing}).
\end{proof}

\begin{coro}
  Not every maximal totally singular subspace is a subalgebra
  of\/~$\OO$.  The group $\Aut\OO$ is not transitive on the set of all
  maximal totally singular subspaces. %
  \qed
\end{coro}

\begin{lemm}\label{neutralElement}%
  If\/~$A$ is a subalgebra of\/~$\OO$ containing an element~$e$ that
  is a neutral element for the multiplication in~$A$ then either
  $e\in\{0,1\}$, or~$\OO$ is split and~$e$ belongs to the unique
  $\Aut[F]\OO$-orbit of elements of trace~$1$ and norm~$0$.
  In the latter case, we have $A=Fe$. 
\end{lemm}

\begin{proof}
  In any case, the subspace $F+Fe$ is an associative and commutative
  subalgebra of~$\OO$. %
  Assume $e\notin\{0,1\}$.
  
  The element~$e$ is an idempotent, and centralizes~$A$. From $e^2=e$
  and $e\notin\{0,1\}$ we infer that the minimal polynomial of~$e$ in
  $F+Fe$ is $X^2-X$. Comparing this with the general form of the
  minimal polynomial $X^2-\tr(a)X+\N(a)$ for any element $a\in\OO$, we
  find $\tr(e)=1$ and $\N(e)=0$. This implies $e\notin F$, and~$e$ is
  a non-trivial element of norm~$0$ in~$\OO$. %
  In particular, the algebra~$\OO$ is split. %
  From~\ref{normAndTraceCharacterizeOrbits} %
  we know that the non-central elements of given norm and trace form a
  single orbit under $\Aut[F]\OO$.

  Now $A$ is contained in the centralizer of the element~$e$ of
  norm~$0$ and trace~$1$. We compute $\N(e-\gal{e}) = -1$ and
  use~\ref{centralizers} to obtain $\dim\C[\OO]e=2$. This yields
  $A \le \C[\OO]e = F+Fe$, and $A = Fe$ follows because
  $A = Ae \le (F+Fe)e = Fe$.
\end{proof}

\goodbreak%
\begin{prop}\label{dim3inPu}
  Let\/~$A$ be a $3$-dimensional subalgebra of\/~$\OO$, and assume
  $A\subseteq 1^\perp$. %
  Then~$\OO$ is split; we write $\OO = F^{2\times2} + F^{2\times2}w$
  as double of the split quaternion algebra $F^{2\times2}$, with
  $w^2=1$. %
  \begin{enumerate}
  \item If $\Char{F}\ne2$ then~$A$ is a Heisenberg algebra.
  \item If $\Char{F}=2$ then either~$A$ is a Heisenberg algebra, or
    $A = F+V$ with a two-dimensional subalgebra~$V$ such that $v'v=0$
    holds for all $v',v\in V$.
  \item The Heisenberg algebras form a single orbit under
    $\Aut[F]\OO$, represented by ${n_0\OO\cap \OO n_0}$.
  \item The subalgebras of the form $F+V$ with a
    two-dimensional subalgebra~$V$ with trivial multiplication form a
    single orbit under $\Aut[F]\OO$, represented by ${F+Fn_0+Fn_0w}$. 
  \end{enumerate}
\end{prop}

\begin{proof}
  In each of the following cases, we will see that $A\smallsetminus\{0\}$
  contains elements of norm~$0$. So~$\OO$ is split. %
  
  Assume first that $\Char{F} \ne 2$. %
  From~\ref{pureAlgebrasIsotropic} we then know that $A$ is totally
  isotropic, and infer that~$A$ is totally singular. %
  If there are $a,b\in A$ such that $ab\ne0$ then $A={Fa+Fb+Fab}$ is a
  Heisenberg algebra, see~\ref{Heisenberg}. So assume that $ab=0$ for
  all $a,b\in A$. Up to an automorphism of~$\OO$
  (see~\ref{normAndTraceCharacterizeOrbits}), we may assume that
  $n_0=\left(
    \begin{smallmatrix}
      0 & 1 \\
      0 & 0
    \end{smallmatrix}\right)$ lies in~$A$. %
  Then $A$ is contained in the intersection of the left annihilator %
  $\set{c+xw}{cn_0=0,x\gal{n_0}=0} =
  F\gal{p_0}+Fn_0+(F\gal{p_0}+Fn_0)w = \OO n_0$ with the right
  annihilator %
  $\set{c+xw}{n_0c=0,x{n_0}=0} = F{p_0}+Fn_0+(F\gal{p_0}+Fn_0)w =
  n_0\OO$, see~\ref{vanderBljSpringer}\ref{kerLambdaA}.  %
  That intersection is the $3$-dimensional subspace
  $n_0\OO\cap \OO n_0 = Fn_0+(F\gal{p_0}+Fn_0)w$, and must coincide
  with~$A$. However, the product $(n_0w)(\gal{p_0}w) = p_0n_0 = n_0$
  is not zero. This contradiction shows that~$A$ is a Heisenberg
  algebra. %

  If $\Char{F}=2$ then $1$ lies in~$1^\perp$, and may thus be a member
  of~$A$. In that case, the restriction of~$\N$ to~$A$ cannot be
  non-singular because otherwise~$A$ would contain an extension field
  of degree two, and~$\dim_FA$ would be even. If $1\notin A$ then
  $\N(A)=\{0\}$ anyway, see~\ref{1inSub}. So we may assume $n_0\in A$
  (see~\ref{normAndTraceCharacterizeOrbits}). %
  We also note that $A\subseteq 1^\perp$ is commutative if
  $\Char{F}=2$ because $xy = \gal{xy} = \gal{y}\gal{x} = yx$ holds for
  all $x,y\in A$. %
  If $1\notin A$ then $\N(A)=\{0\}$, and we invoke~\ref{Heisenberg} as
  above to obtain that $A$ is a Heisenberg algebra. %
  If $1\in A$ %
  then $A\subseteq \C[\OO]{n_0} = F+Fn_0+F^{2\times2}w$, and there
  exists $y\in F^{2\times2}\smallsetminus\{0\}$ with $yn_0\in Fy$ such
  that $A=F+Fn_0+Fyw$. Then $y\in Fn_0+F\gal{p_0}$, and we obtain
  $0 = n_0^2 = (yw)^2 = n_0(yw) = (yw)n_0$. %
  In particular, the subspace $Fn_0+Fyw$ forms a subalgebra with
  trivial multiplication.

  It remains to show that the isomorphism type of the subalgebras in
  question determines their orbit under $\Aut[F]\OO$. %
  If~$A$ is a Heisenberg algebra, we may (up to an automorphism
  of~$\OO$) assume that $n_0$ lies in~$A$, and %
  that $n_0A = \{0\} = An_0$.  We then find that $A$ coincides with
  the two-sided annihilator
  $n_0\OO\cap\OO n_0 = Fn_0+(F\gal{p_0}+Fn_0)w$,
  see~\ref{vanderBljSpringer}\ref{kerLambdaA}. %
  Therefore, these Heisenberg algebras are in a single orbit. %

  Finally, consider a two-dimensional subalgebra $V$ with trivial
  multiplication. %
  Every element of $V\smallsetminus\{0\}$ is in the $\Aut[F]\OO$-orbit
  of~$n_0$, we may (and will) assume $n_0\in V$. %
  Then $V = Fn_0+Fy$, where~$y$ belongs to the two-sided annihilator %
  $n_0\OO\cap\OO n_0 = {Fn_0+(F\gal{p_0}+Fn_0})w$ of~$n_0$. If
  $n_0\in V \not\subseteq Fn_0+Fn_0w$ then~$V$ contains an element of
  the form $(\gal{p_0}+rn_0)w$, with $r\in F$. Using $s = 1$ and
  $t \coloneqq \left(
    \begin{smallmatrix}
      0 & 1 \\
      -1 & r
    \end{smallmatrix}
  \right) \in \HH$,
  we obtain an automorphism $\alpha_{1,t}\colon a+xw \mapsto a+(tx)w$
  of~$\OO$ with $\alpha_{1,t}(n_0)=n_0$ and
  $\alpha_{1,t}((\gal{p_0}+rn_0)w) = n_0w$, as required. 
\end{proof}

\begin{lemm}\label{pnOrbit}
  Let\/ $p,n$ be two non-trivial elements  (necessarily of norm~$0$) %
  with $p^2=p\ne1$, $pn=n$,
  $n^2=0=np$. Then $Fp+Fn$ is a  subalgebra. These subalgebras form a
  single orbit under~$\Aut[F]\OO$. %
\end{lemm}
\begin{proof}
  The vector subspace $E \coloneqq F+Fp$ is a (split) composition
  algebra. The subspace $\{1,p,n\}^\perp$ has dimension~$5$, and
  contains an element~$u$ with $\N(u)\ne0$. So $H\coloneqq E+Eu$ is a (split)
  quaternion algebra. %
  We note that $(pu|n) = (u|\gal{p}n) = (u|(1-p)n) = (u|0) = 0$, and
  infer that~$n$ lies in~$H^\perp$. 

  From~\ref{SkolemNoether} we know that there is $\alpha\in\Aut[F]\OO$
  with $\alpha(H) = \HH $ (${}=F^{2\times2}$).  %
  Then $\alpha(p)$ is a conjugate of $p_0 \coloneqq \left(
    \begin{smallmatrix}
      1 & 0 \\
      0 & 0 
    \end{smallmatrix}\right)$ in~$\HH $, and the corresponding inner
  automorphism of~$\HH $ extends to an automorphism~$\beta$
  of~$\OO$. So we may assume $p=p_0$ and $H=\HH $, while
  $n\in \HH ^\perp$. Then there exists $y\in \HH $ such that $n=yw$,
  and $0 = np = (yw)p = (y\gal{p})w$ yields $y\gal{p} = 0$. This means
  that $y=\left(
    \begin{smallmatrix}
      y_1 & 0 \\
      y_2 & 0 
    \end{smallmatrix}\right)$. %
  There exists $s\in \SL[2]F \subseteq \HH $ such that
  $s\binom{y_1}{y_2} = \binom10$, and $a+xw \mapsto a+(sx)w$ is an
  automorphism of~$\OO$ mapping $n = yw$ to~$p_0w$.
\end{proof}

\begin{exam}
  The elements $p_0 \coloneqq \left(
    \begin{smallmatrix}
      1 & 0 \\
      0 & 0 
    \end{smallmatrix}\right)$ and %
  $n_0 \coloneqq \left(
    \begin{smallmatrix}
      0 & 1 \\
      0 & 0
    \end{smallmatrix}\right)$ %
  in $F^{2\times2}$ satisfy the assumptions in~\ref{pnOrbit}; they
  generate the subalgebra %
  $L \coloneqq Fp_0+Fn_0 =   \set{\left(
      \begin{smallmatrix}
        r & s \\
        0 & 0 
      \end{smallmatrix}\right)}{r,s\in F}$ of~$F^{2\times2}$. %
  
  This subalgebra~$L$ is not invariant under the standard involution;
  its image under~$\kappa$ is %
  $\gal{L} = \gal{Fp_0+Fn_0} = \set{\left(
      \begin{smallmatrix}
        0 & s \\
        0 & r 
      \end{smallmatrix}\right)}{r,s\in F}$.   %
  The algebras $L$ and $\gal{L}$ are not isomorphic.  %
\end{exam}

\begin{lemm}\label{mnOrbit}
  Let\/ $m,n$ be two non-trivial elements of\/~$1^\perp$ with
  $\N(m)=0=\N(n)$ and $mn=0$. Then $nm=0$, and $Fm+Fn = m\OO\cap n\OO$
  is a subalgebra of dimension~$2$. %
  These subalgebras form a single orbit under~$\Aut[F]\OO$, and for
  any such algebra~$A$, the stabilizer of\/~$A$ in~$\Aut[F]\OO$ acts
  two-transitively on the set of one-dimensional subspaces of\/~$A$.
\end{lemm}
\begin{proof}
  We know from~\ref{maximalIsotropicSubalgebras} that~$m\OO$
  and~$n\OO$ are subalgebras, so $m\OO\cap n\OO$ is a subalgebra, as
  well. As $m,n$ lie in~$1^\perp$, we have $nm = \gal{mn} =0$. 

  Using~\ref{normAndTraceCharacterizeOrbits}, we may assume
  $n = n_0 \coloneqq \left(
    \begin{smallmatrix}
      0 & 1 \\
      0 & 0 
    \end{smallmatrix}\right) \in \HH $. %
  There are $b,y\in \HH $ such that $m=b+yw$, and $0 = mn_0 =
  bn_0-(y{n_0})w$ yields $bn_0=0$ and $yn_0=0$. On the other hand, we have
  $0 = n_0m = n_0b+(yn_0)w$, and obtain $n_0b=0$, as well. This yields $b\in
  Fn_0$ and $y=\left(
    \begin{smallmatrix}
      0 & y_1 \\
      0 & y_2
    \end{smallmatrix}\right)$ with $y_1,y_2\in F$. %
  Now $Fm+Fn_0 = F(m-b)+Fn_0$, there exists
  $s\in \SL[2]F \subseteq \HH $ such that $sy=n_0$, and
  $a+xw \mapsto a+(sx)w$ is an automorphism of~$\OO$ mapping
  $m-b = yw$ to~$n_0w$.

  For the last assertion, it now suffices to consider
  $A = \set{\left(
      \begin{smallmatrix}
        r & s \\
        0 & 0
      \end{smallmatrix}\right)w}{r,s\in F}$.
  The stabilizer of~$A$ in~$\Aut[F]\OO$ contains the maps
  $a+xw \mapsto s^{-1}as+(xs)w$, where $s\in\SL[2]F < \HH $. Clearly,
  these maps form a group that acts two-transitively on the set of
  one-dimensional subspaces of~$A$, as claimed.
\end{proof}

\begin{coro}\label{twoTrsPointRow}
  The group~$\Aut[F]\OO$ acts transitively on the set of all pairs
  $(Fm,Fn)$ of different subspaces spanned by non-trivial elements
  $m$, $n$ in~$\OO$ with $\N(m) = \N(n) = \tr(m) = \tr(n) = 0$ and
  $mn=0$. %
  \qed
\end{coro}

\begin{nthm}[Totally singular subalgebras]%
  \label{allTotallySingularSubalgebras}%
  Let\/~$A$ be a subalgebra of\/~$\OO$, and assume $1\notin
  A$. Then~$A$ is totally singular, and\/ $\dim{A}\le4$.  In order to
  give representatives for the orbits of subalgebras, we write
  $\OO = \HH +\HH w$ as the double of the split quaternion algebra
  $\HH \coloneqq F^{2\times2}$ with $w$ such that $w^2=1$, and use
  $p_0 \coloneqq \left(
        \begin{smallmatrix}
          1 & 0 \\
          0 & 0
        \end{smallmatrix}\right)$ %
      and $n_0 \coloneqq \left(
        \begin{smallmatrix}
          0 & 1 \\
          0 & 0
        \end{smallmatrix}\right)$. 
  \begin{enumerate}
  \item If\/ $\dim{A}=1$ then $A=Fa$ with $a\in\OO$ such that
    $\N(a)=0$. These subalgebras form two orbits under~$\Aut[F]\OO$,
    represented, respectively, by~$Fp_0$, and by~$Fn_0$. %
    Note that\/~$F$ and\/~$Fp_0$ are isomorphic as algebras, %
    under an isomorphism that maps $1\in F$ to
    $p_0\in\OO\setminus\{1\}$.  %
  \item Totally singular subalgebras of dimension~$2$ form three
    orbits under~$\Aut[F]\OO$:
    \begin{enumerate}
    \item Algebras with a left identity for the multiplication, such
      as $Fp_0+Fn_0$ (see~\ref{pnOrbit}).
    \item Algebras with a right identity for the multiplication,
      obtained by applying some anti-automorphism (such as~$\kappa$)
      to algebras with left identity (see~\ref{pnOrbit}).
    \item Algebras with trivial multiplication, such as $(Fp_0+Fn_0)w$
      (see~\ref{mnOrbit}).
    \end{enumerate}
    Each one of the algebras in the latter orbit is contained
    in~$1^\perp$, while none of those in the first two of these orbits
    is contained in~$1^\perp$.
  \item Totally singular subalgebras of dimension~$3$ form two
    orbits under~$\Aut[F]\OO$, namely:
    \begin{enumerate}
    \item Heisenberg algebras $Fx+Fy+Fn$ with $xy=n=-yx$, obtained as
      $n\OO\cap \OO n$ for $n\in\OO$ such that $n^2=0\ne n$. %
      These subalgebras are contained in~$1^\perp$.
    \item\label{mOOcapOOn}%
      Algebras of the form $m\OO \cap \OO n$ where $m,n\in\OO$ are
      linearly independent, with $m^2 = 0 = n^2=0 = mn = nm$.  %
      None of these subalgebras is contained in~$1^\perp$.
    \end{enumerate}
  \item\label{totSing4orbits}%
    Totally singular subalgebras of dimension~$4$ form two
    orbits under~$\Aut[F]\OO$, namely:
    \begin{enumerate}
    \item $n\OO$  for $n\in\OO$ such that $n^2=0\ne n$.
    \item $\OO n$  for $n\in\OO$ such that $n^2=0\ne n$.
    \end{enumerate}
  \end{enumerate}
  \goodbreak\noindent %
  In particular, two totally singular subalgebras of\/~$\OO$ are
  isomorphic if, and only if, they are in the same orbit
  under~$\Aut[F]\OO$. 

  The totally singular subalgebras of dimension at most~$3$ are
  associative. The totally singular subalgebras of dimension~$4$ are
  not associative.
\end{nthm}
\begin{proof}
  By~\ref{1inSub}, our assumption $1\notin A$ implies $\N(A)=\{0\}$,
  and~$A$ is totally singular. %
  As the quadratic form~$\N$ is not degenerate, we have
  $\dim{A} \le \frac12\dim\OO = 4$.

  If $\dim{A}=1$ then $A=Fa$ with $a\in\OO$ with $\N(a)=0$. If
  $\tr(a)\ne0$ then $\tr(a)^{-1}a$ has trace~$1$, and lies in the
  orbit of~$p_0$ under~$\Aut[F]\OO$
  by~\ref{normAndTraceCharacterizeOrbits}. %
  If $\tr(a)=0$ then $a$ lies in the orbit of~$n_0$
  under~$\Aut[F]\OO$.

  Now assume that $A=Fx+Fy$ has dimension~$2$. Then~$A$ is
  associative, and there are $r,s\in F$ such that $xy=rx+sy$. %
  Without loss, we may assume $y\in 1^\perp$. Then $y^2=0$, and
  $0 = xy^2 = (xy)y = rxy + sy^2 = rxy$ yields $xy\in Fy$. 
  
  If $A \subseteq 1^\perp$, we also have  $x^2=0$, and $xy=0$
  follows. In this case, the algebra $A$ is in the orbit of $(Fp+Fn)w
  = Fpw+Fnw$  under~$\Aut[F]\OO$, see~\ref{mnOrbit}. 
  
  If $A\not\subseteq 1^\perp$, we may assume $\tr(x)=1$. Then $x^2=x$
  yields $sy = xy = x(xy) = s^2y$, and $s\in\{0,1\}$ follows. %
  From $0 = (x|y) = \gal{x}y+\gal{y}x = (1-x)y-yx$ we then infer
  $yx = y-xy = y-sy$.  Thus $A$ is in the orbit of $Fp+Fn$ if $s=1$,
  and in the orbit of $F\gal{p}+Fn$ if $s=0$, see~\ref{mnOrbit}.

  Now assume $\dim{A}=3$. Then there are $x,y\in\OO$ with
  $\N(x)=0=\N(y)$ and $xy=0$ such that $A= x\OO\cap\OO y$. %
  From~\ref{subalgebrasInpOO} we infer $x,y\in 1^\perp$. %

  If $Fx=Fy$ then $A$ is in the orbit of $n_0\OO\cap\OO n_0$
  by~\ref{normAndTraceCharacterizeOrbits}, and
  $n_0\OO\cap\OO n_0 = Fn_0+({F\gal{p_0}+Fn_0})w \subseteq 1^\perp$. %
  If $Fx\ne Fy$ then~\ref{twoTrsPointRow} yields that $A$ is in the
  orbit of $(p_0w)\OO\cap\OO(n_0w) = Fp_0+(Fp_0+Fn_0)w$,
  cp.~\ref{threedimPureSubalgebra}\ref{pownow}. That subalgebra
  contains $p_0\notin 1^\perp$. %
  We know from~\ref{threedimPureSubalgebra} that these subalgebras are
  associative. %

  Finally, assume $\dim{A}=4$. Then either $A=x\OO$ or $A=\OO x$
  holds for some $x\in\OO$ with $\N(x)=0$, and~\ref{subalgebrasInpOO}
  yields $\tr(x)=0$. 
  Thus either~$A$ is in the orbit of~$n_0\OO$, or~$A$ is in the orbit
  of~$\OO n_0$, see~\ref{normAndTraceCharacterizeOrbits}. %
  We know from~\ref{fourdimNotAssoc} that these subalgebras are not
  associative. %
\end{proof}

\section{Subalgebras that are not totally singular}

\begin{lemm}\label{radicalQuotient}
  Assume that\/~$A$ is a subalgebra of\/~$\OO$. Consider the
  radicals\/ $R\coloneqq A^\perp\cap A$ of the restriction of the
  polar form, and\/ $Q \coloneqq \set{x\in R}{\N(x)=0}$
  of the restriction~$\N|_A$ of the quadratic form, respectively. %
  \begin{enumerate}
  \item\label{Rideal}%
    We have $AR \subseteq R$, and $RA \subseteq R$. 
  \item In any case, both~$Q$ and\/~$R$ are subalgebras of\/~$\OO$.
  \item If the restriction of the polar form to~$A$ is not trivial
    then~$R$ is totally singular. So $R = Q$, and\/ $A/R$ is a
    composition algebra.
  \item If the restriction of the polar form to~$A$ is trivial then
    $A = R$. So either~$A=Q$ is totally singular, or $R\ne Q$. In the
    latter case, we have $\Char{F}=2$, the algebra~$A$ is commutative,
    and contains\/~$F$.
  \end{enumerate}
\end{lemm}
\begin{proof}
  If~$A$ contains no invertible element, then the restriction of the
  norm form to~$A$ is zero, and so is the restriction of the polar
  form. Then $R=A$ and assertion~\ref{Rideal} is trivial.

  So assume that~$A$ contains invertible elements; then $1\in A$
  and~$A$ is closed under~$\kappa$ by~\ref{closedUnderKappa}. For
  $x\in R$ and $a,b\in A$, we compute $(ax|b) = (x|\gal{a}b) = 0$. So
  $ax\in R$, and $AR \subseteq R$ follows. The argument for
  $RA\subseteq R$ is completely symmetric to the one just given.  So
  assertion~\ref{Rideal} is established, and we know that $R$ is as
  subalgebra. %
  
  In any case, the set~$Q$ is closed under multiplication. %
  If $\Char{F}\ne2$ then $Q=R$. %
  If $\Char{F}=2$ then the restriction $\N|_R\colon R\to F$ is a
  semi-linear map (with respect to the Frobenius endomorphism),
  and~$Q$ is the kernel of that map. So~$Q$ is a vector subspace, and
  then a subalgebra.

  Now assume that the restriction of the polar form to~$A$ is not
  trivial. %
  We show first that~$R$ consists of elements of norm~$0$; i.e.,
  $R=Q$. %
  By our assumption, there exist $a,b\in A$ such that $(a|b)\ne0$. %
  For $x\in R$ we use $\gal{x},b\gal{x},b\gal{x}x\in R$ and compute
  $0 = (a|b\gal{x}x) = (a|b)x\gal{x}$.  %
  This gives $\N(x)=0$, as claimed.
  
  Using assertion~\ref{Rideal}, it is straightforward to check (by the
  usual arguments known for associative algebras) that addition and
  multiplication on $A/R$ are well-defined by $(a+R)(b+R) = (a+b)+R$,
  and $(a+R)(b+R) = (ab)+R$. The induced norm $M(a+R) \coloneqq \N(a)$
  is well-defined, and multiplicative on~$A/R$. The polar form of~$M$
  is non-degenerate, and $1\in A$ follows from the fact that the
  subalgebra~$A$ is not totally singular (see~\ref{1inSub}). %
  So indeed $A/R$ is
  a composition algebra.

  Finally, assume that the restriction of the polar form to~$A$ is
  trivial, i.e., $R=A$.  %
  If $Q \ne R$ then~$R$ contains an invertible element~$b$. Then
  $0 = (b|b) = 2\N(b)$ yields $\Char{F} = 2$, and
  $b^2 = \N(b) \in (F\cap R) \smallsetminus\{0\}$ shows
  $F\subseteq R$.  For $x,y \in R$, we now have
  $0 = (x|y) = \gal{x}y + \gal{y}x = -xy-yx$, and use $\Char{F}=2$ to
  infer that~$R$ is commutative.
\end{proof}

\begin{nthm}[Subalgebras of dimension greater than four]%
  \label{dimGreaterFour}%
  Let\/ $A$ be a subalgebra of\/~$\OO$. %
  If\/ $\dim{A}>4$ then either $A=\OO$, or~$\OO$ is split and one of
  the following cases occurs:
  \begin{enumerate}
  \item\label{dimFive}%
    $\dim{A}=5$, and there exists a nilpotent element\/~$n$ such that
    $A=n\OO+\OO n$. These subalgebras form a single orbit
    under~$\Aut[F]\OO$. 
  \item\label{dimSix}%
    $\dim{A}=6$, and there exist two linearly independent nilpotent
    elements\/~$m,n$ with $mn=0$ such that $A=\{m,n\}^\perp$.  %
    Then $A = m\OO+n\OO = \OO m+\OO n$.  These subalgebras form a
    single orbit under~$\Aut[F]\OO$.
  \end{enumerate}
\end{nthm}
\begin{proof}
  We know that $1\in A$ because $\dim{A}>4$ implies that~$A$ is not
  totally singular.  From~\ref{pureAlgebrasIsotropic} we know
  $A\not\subseteq 1^\perp$. %
  We show first that there is a two-dimensional subalgebra
  $B\subseteq A$ such that the restriction of the polar form to~$B$ is
  not degenerate; then~$B$ is a composition algebra. %

  If $\Char{F}=2$ we just pick $b\in A$ such that $b\notin
  1^\perp$. Then $b\notin F\subseteq 1^\perp$, and
  $B\coloneqq F+Fb \subseteq A$ is a two-dimensional composition
  algebra, as required.

  If $\Char{F}\ne2$ we first note that $1^\perp\cap A$ is not totally
  singular. %
  In fact, that intersection has at least dimension~$4$. The
  restriction of the polar form to~$1^\perp$ is not degenerate (here
  we use $\Char{F}\ne2$), and does not contain any totally isotropic
  subspace of dimension greater than~$3$. %
  So there exists $a\in 1^\perp\cap A$ such that $\N(a)\ne0$, and
  $B\coloneqq F+Fa \subseteq A$ is a composition algebra, as
  required. %
  
  Now consider $B^\perp\cap A$. If there exists an invertible element
  $u\in B^\perp\cap A$ then $H \coloneqq B+Bu$ is a quaternion
  subalgebra in~$A$. %
  If there is still an invertible element $v\in H^\perp\cap A$ then
  $\OO = H+Hv$ coincides with~$A$. %
  So we may assume that $H^\perp\cap A$ consists of nilpotent
  elements. %
  Then the polar form vanishes on that intersection, and we find that
  $H^\perp \cap A$ equals the radical $R \coloneqq A^\perp\cap A$
  of~$A$. %
  Choose $w\in H^\perp$ with $\N(w)=-1$. Then $\OO = H+Hw$ is obtained
  by doubling, and $R\subseteq Hw$. %
  For $a,x\in H$, we find $a(xw) = (xa)w$, and~$R$ is a non-trivial
  module over the opposite algebra of $H \cong F^{2\times2}$. This
  yields $\dim{R} \ge2$, and any two-dimensional subspace
  $L\subseteq R$ gives $A \le R^\perp \le L^\perp$. So $A= L^\perp$
  and $L=R$ hold for dimension reasons. %
  The multiplication formula~\ref{propertiesCA}.\ref{doubling} shows
  $H^\perp H^\perp \subseteq H$, and~\ref{radicalQuotient}\ref{Rideal}
  yields $RR \subseteq R\cap H \subseteq H^\perp\cap H = \{0\}$. %
  So the multiplication in~$R$ is trivial, and~$R$ is in the orbit of
  $(Fp_0+Fn_0)w$, see~\ref{mnOrbit}
  and~\ref{allTotallySingularSubalgebras}. We obtain that~$A$ is in
  the orbit of $\left((Fp_0+Fn_0)w\right)^\perp = H+(Fp_0+Fn_0)w$,
  which is a subalgebra by~\ref{rightIdealConstruction}
  (see~\ref{Exasleftidealconstruction}).

  It remains to consider the case where $B^\perp\cap A$ consists of
  nilpotent elements. Then the norm vanishes on that subspace, and
  $Q \coloneqq B^\perp\cap A$ is the radical of the quadratic
  form~$\N|_A$.  As~$Q$ is a totally singular subspace of~$1^\perp$,
  we have $\dim{Q}\le3$ and then $\dim{Q}=3$ and $\dim{A}=5$.  The
  subspace~$Q$ forms a totally singular subalgebra of dimension~$3$,
  and is contained
  in~$1^\perp$. From~\ref{allTotallySingularSubalgebras} we know
  that~$Q$ is in the orbit of $n_0\OO\cap\OO n_0$, and $A = Q^\perp$
  is in the orbit of $(n_0\OO\cap\OO n_0)^\perp = n_0\OO+\OO n_0$.
\end{proof}

\begin{lemm}\label{QhasDim2}
  Let\/~$A$ be a $3$-dimensional subalgebra with $1\in A$ and
  $\dim{R}\ge2$. Then $\dim{Q}\ge2$. 
\end{lemm}
\begin{proof}
  We note first of all that~$A$ does not contain any $2$-dimensional
  vector subspace~$V$ such that the restriction~$\N|_V$ of the norm
  form is non-singular: for any $v\in V\setminus\{0\}$, the subspace
  $v^{-1}V$ would be a subalgebra without zero divisors (i.e., a
  quadratic extension field of~$F$), and the dimension of~$A$ would be
  even. 
  
  If $\dim{Q}<2$ then $Q\ne R$, and $\Char{F}=2$ follows. Consider
  $y\in R\setminus Q$. %
  If $y\in R$ then $A\subseteq 1^\perp$, and $A=R$ holds
  by~\ref{pureAlgebrasIsotropic}. If $y\notin F$ then $\Char{F}=2$
  yields $1\perp1$, and we obtain $A=R$, again. %
  Now each singular element of~$A$ lies in~$Q$. We pick a vector space
  complement~$V$ for~$Q$ in~$A$ and reach the contradiction
  that~$N|_V$ is non-singular.
\end{proof}

After~\ref{allTotallySingularSubalgebras} and~\ref{dimGreaterFour}, it
remains to treat subalgebras that are not totally
singular (so they contain~$1$) and have dimension at most~$4$.

\begin{nthm}[Unary subalgebras of dimension at most four]%
  \label{unaryDimAtMostFour}%
  Let\/ $A$ be a subalgebra of an octonion algebra~$\OO$, with
  $1\in A$.  If\/~$\OO$ is split, we write $\OO = \HH+\HH w$, with
  $\HH = F^{2\times2}$ and\/~$w^2=1$.
  \begin{enumerate}
  \item\label{dim1}%
    If\/ $\dim{A} = 1$ then $A = F$. 
  \item\label{dim2}%
    If\/ $\dim{A} = 2$ then we have one of the following cases:
    \begin{enumerate}
    \item $A$ is a separable quadratic extension field over~$F$.
    \item $A$ is an inseparable quadratic extension field over~$F$
      (and\/ $\Char{F}=2$).
    \item\label{FxF}%
      $A \cong F[X]/(X^2-X) \cong F\times F$ is a split
      composition algebra.
    \item $A \cong F[X]/(X^2)$ is a local algebra.
    \end{enumerate}
  \item\label{dim3}%
    If\/ $\dim{A} = 3$ then the radical $R\coloneqq A^\perp\cap A$
    is not trivial, and one of the following occurs.
    \begin{enumerate}
    \item\label{triangularMatrices}%
      $\dim{R}=1$, and\/ $A$ is in the $\Aut[F]\OO$-orbit of the
      subalgebra $F+Fp_0+Fn_0$ of upper triangular matrices in~$\HH $
      (see~\ref{pnOrbit}).
    \item $\dim{R}\ge2$, and\/~$A$ is in the orbit of the subalgebra
      $F+(Fp_0+Fn_0)w$ (see~\ref{mnOrbit}).
    \end{enumerate}
  \item\label{dim4}%
    If\/ $\dim{A} = 4$ then one of the following occurs:
    \begin{enumerate}
    \item $A$ is a quaternion algebra (split or not).
    \item\label{QtwoDim}%
      $A$ is in the orbit of $B+ (Fp_0+Fn_0)w$, where %
      $B \subseteq \HH$ %
      is either a two-dimensional composition algebra, or an
      inseparable quadratic field extension. %
    \item\label{F+Heisenberg}%
      The radical\/
      $Q \coloneqq \set{x\in A^\perp \cap A}{\N(x)=0}$ of the
      restriction~$\N|_A$ of the quadratic form has dimension~$3$. %
      Then~$Q$ is isomorphic to the Heisenberg algebra, and\/ $A$ is
      in the orbit of\/ $F+(n_0\OO\cap\OO n_0)$. %
    \item\label{totallyInseparable}%
      $A = A^\perp$ is a totally inseparable field extension of
      degree~$4$ and exponent~$1$ (and\/ $\Char{F}=2$). %
      \\ %
      Of course, this case can only occur if\/ $\Char{F}=2$.  Two such
      subfields are in the same orbit under $\Aut[F]\OO$ if, and only
      if, they are isomorphic as $F$-algebras,
      see~\upshape{\cite[5.7]{MR2746044}}.
    \end{enumerate}
  \end{enumerate}
  In each of these cases, subalgebras are in the same orbit under
  $\Aut[F]\OO$ precisely if they are isomorphic as $F$-algebras %
  (and contain~$1$, as assumed here).

  Each subalgebra~$A$ with $1\in A$ and $\dim{A}\le4$ is associative.
\end{nthm}
\begin{proof}
  If $\dim{A}=2$, pick $a\in A\smallsetminus F$. Then $1,a$ is a basis
  for~$A$, and the minimal polynomial of~$a$ is
  $m_a(X) = X^2-\tr(a)X+\N(a)$. The four cases in assertion~\ref{dim2}
  now belong to the cases where $m_a(X)$ is irreducible and separable
  (viz., $\Char{F}\ne2$ or $\tr(a)\ne0$), or $m_a(X)$ is irreducible
  and inseparable (viz., $\Char{F}=2$ and $\tr(a)=0$), or $m_a(X)$ is
  reducible with two different roots in~$F$, or $m_a(X)$ has a double
  root in~$F$, respectively. %
  Two such algebras $A$ and~$A'$ are in the same orbit under
  $\Aut[F]\OO$ if, and only if, there are elements $a\in A\smallsetminus F$
  and $a'\in A'\smallsetminus F$ such that~$a$ and~$a'$ have the same norm
  and trace (see~\ref{normAndTraceCharacterizeOrbits}). In that case,
  the algebras~$A$ and~$A'$ are both isomorphic to $F[X]/(m_a(X))$.

  If $\dim{A}=3$ then~$A$ is not a composition algebra. Therefore, the
  restriction of the polar form is degenerate, and the radical
  $R \coloneqq A^\perp\cap A$ is not
  trivial. From~\ref{radicalQuotient}\ref{Rideal} we know
  $AR = R = RA$.

  Assume that $\dim{R}=1$, and choose a vector space complement $B$
  for~$R$ in~$A$ such that $1\in B$. Then~$B$ is a subalgebra, and the
  restriction of the polar form to~$B$ is not degenerate.  So~$B$ is a
  composition algebra, and~$R$ is a $B$-module. As $\dim{R}=1$, the
  annihilator $\set{a\in B}{aR=\{0\}}$ is not trivial. Therefore, the
  algebra~$B$ is not a field but a split composition algebra, and~$R$
  does not contain invertible elements (see
  also~\ref{radicalQuotient}). Then $B\cong F\times F$, and each
  non-trivial element of $1^\perp\cap B$ is invertible. %
  Choose $a\ne0$ in the annihilator of~$R$ in~$B$. Then
  $\N(a)=0\ne\tr(a)$, and $M \coloneqq Fa+R$ is a totally singular
  subalgebra containing an
  idempotent. From~\ref{allTotallySingularSubalgebras} we now infer
  that~$M$ is in the orbit of either $Fp_0+Fn_0$ or
  $\kappa(Fp_0+Fn_0)$. In any case, the algebra $A=F+M$ is in the
  orbit of the algebra $F+Fp_0+Fn_0 = F+\kappa(Fp_0+Fn_0)$ of upper
  triangular matrices, as claimed.
  
  Now assume $\dim{R}\ge2$ (and still $\dim{A}=3$). %
  Then the radical~$Q$ of~$\N|_A$ has dimension~$2$ by~\ref{QhasDim2},
  and forms a totally singular subalgebra, see~\ref{radicalQuotient}.
  From~\ref{allTotallySingularSubalgebras} we conclude that~$Q$ is in
  the orbit of $(Fp_0+Fn_0)w$. So~$A = F+Q$ is in the orbit of
  $F+(Fp_0+Fn_0)w$, as claimed.

  Finally, consider the case where
  $\dim{A}=4$. From~\ref{radicalQuotient} we know that $A/R$ is a
  composition algebra if $R\ne A$. So $\dim{R} = 4 -\dim(A/R) \in
  \{0,2,3,4\}$.  %
  If $\dim{R} = 0$ then $A$ is a
  composition algebra, and then a quaternion algebra.

  If $\dim{R} = 2$ then the totally singular subalgebra $R$ is in the
  orbit of $(Fp_0+Fn_0)w$, see~\ref{allTotallySingularSubalgebras},
  and we assume $R=(Fp_0+Fn_0)w$ without loss of generality. %
  We have $A \subseteq \{p_0w,n_0w\}^\perp = \HH+R$, and
  $B \coloneqq A\cap\HH$ has dimension~$2$. As~$A$ contains invertible
  elements, we have $1\in B$. So~$B$ is a subalgebra. From
  $B\cap R=\{0\}$ we infer that the restriction of the polar form
  to~$B$ is not degenerate. So~$B$ is a composition algebra (either a
  separable quadratic extension field of~$F$, or isomorphic to
  $F\times F$), and $A=B+ (Fp_0+Fn_0)w$. %

  Consider a four-dimensional subalgebra~$A'$ with
  $R\subseteq A' \subseteq \HH+R$. Then~$A$ and~$A'$ are isomorphic as
  $F$-algebras if, and only, the subalgebras $B$ and
  $B'\coloneqq A'\cap\HH$ are isomorphic. In that case, we find
  $b\in B\smallsetminus F$ and $b'\in B'\smallsetminus F$ with the same norm
  (i.e., the same determinant) and the same trace. So there exists
  $s\in\GL[2]F = \HH^\times$ with $sbs^{-1}=b'$. Pick an upper
  triangular matrix $t\in\HH$ (${}=F^{2\times2}$) with $\det t=\det s$.
  The map $\alpha_{s,t}\colon a+xw \mapsto sas^{-1}+(txs^{-1})w$ is an
  $F$-linear automorphism of~$\OO$, see~\ref{stabilizerH}.  Now
  $\alpha_{s,t}(b)=b'$, and it is easy to verify
  $\alpha_{s,t}(R)=R$. So $\alpha_{s,t}(A) = A'$. This settles those
  cases in assertion~\ref{QtwoDim} where $\dim{R}=2$. %
  The remaining cases that belong to assertion~\ref{QtwoDim} will be
  discussed below; those are cases where $R=A$ but $\dim{Q}=2$.

  Now assume $\dim{R} \ge 3$, and consider the radical
  $Q \coloneqq \set{x\in A^\perp \cap A}{\N(x)=0}$ of the
  restriction~$\N|_A$ of the quadratic form. %
  From~\ref{radicalQuotient} we know that~$Q$ is a subalgebra.  As
  $1\notin Q$, we have $\dim{Q} \le 3$. %

  If $\dim{Q} = 3$ then the totally singular subalgebra~$Q$ is in the
  orbit of the Heisenberg algebra $n_0\OO\cap\OO n_0$
  (see~\ref{allTotallySingularSubalgebras}), and
  assertion~\ref{F+Heisenberg} follows.

  If $\dim{Q} < 3$ then $Q \ne R$, and $\Char{F} = 2$. In particular,
  the polar form is alternating, and $R = A$ because the co-dimension
  of~$R$ in~$A$ is even.  From~\ref{radicalQuotient} we know that~$A$
  is commutative.
  
  If $\dim{Q}=0$, we pick $b\in A\smallsetminus F$, and note that
  $B\coloneqq F+Fb$ is an inseparable quadratic field extension
  of~$F$. For any $c\in A\smallsetminus B$, we then have
  $B\cap Bc =\{0\}$, and $A = B+Bc$ shows that the algebra~$A$ is
  generated by $\{1,b,c\}$. So~$A$ is associative, commutative, and
  thus a totally inseparable field extension of~$F$ (of degree~$4$ and
  exponent~$1$), as claimed in~\ref{totallyInseparable}.

  If $\dim{Q}\ne0$ then there are $u,v\in A\smallsetminus Q$ such that
  $F+Fu+Fv$ contains a vector space complement of~$Q$ in~$A$. Then the
  subalgebra spanned by $\{1,x,y\}$ is associative, and the
  quotient~$A/Q$ is a totally inseparable field extension of~$F$. In
  particular, the degree of that extension is a power of~$2$, and we
  are left with the case that $\dim{Q}=2$.
  Now~\ref{allTotallySingularSubalgebras} yields that~$Q$ is in the
  orbit of $(Fp_0+Fn_0)w$, and we may assume $Q = (Fp_0+Fn_0)w$.  Then
  $A\subseteq \HH+Q$, and $B \coloneqq A\cap\HH$ is a two-dimensional
  subalgebra, again (albeit not a composition algebra, but an
  inseparable quadratic extension field).

  If~$A'$ is another such algebra (contained in $\HH+Q$), we note
  that~$A$ is isomorphic to~$A'$ precisely if $A/Q$ is isomorphic
  to~$A'/Q$. So $A\cong A'$ means $B\cong B' \coloneqq A'\cap\HH$. As
  above, any algebra isomorphism between $B$ and~$B'$ extends to an
  algebra automorphism $a\mapsto sas^{-1}$ of~$\HH$, any such
  automorphism extends to an automorphism $\alpha_{s,t}$ of~$\OO$ such
  that $\alpha_{s,t}(Q) = Q$, and $\alpha_{s,t}(A) = A'$ follows. %
  We have thus also completed the proof of assertion~\ref{QtwoDim}.

  It remains to verify that each subalgebra~$A$ with $1\in A$ and
  $\dim{A}\le4$ is associative. For the cases in~\ref{dim1} and
  in~\ref{dim2}, this is obvious. %
  For the cases in~\ref{dim3}, we note that $A$ is of the form $A =
  F+X$, where $X$ is a subalgebra generated by two elements. So~$X$ is
  associative  by Artin's theorem (see~\ref{propertiesCA}\ref{artin}),
  and~$A$ is associative, as well. %

  In~\ref{dim4}, the quaternion algebras and the field extensions are
  associative. %
  For the algebra $F+(n_0\OO\cap\OO n_0)$ we use the fact that the
  subalgebra $n_0\OO\cap\OO n_0$ is associative (it is a Heisenberg
  algebra, generated by~$n_0w$ and~$\gal{p_0}w$,
  see~\ref{threedimPureSubalgebra}).

  In the last remaining case, we have to consider $A = B+X$, where
  $B = F+Fb \subseteq \HH$ %
  is a two-dimensional (associative, and commutative) subalgebra
  of~$\HH = F^{2\times2}$, and $X=(Fp_0+Fn_0)w$. %
  Since $(\gal{yw})(xw) = 0$ holds for all $x,y\in X$, the
  multiplication in~$A$ is given by
  $(a+xw)(b+yw) = ab + (ya+x\gal{b})w$, where $a,b\in B$ and
  $x,y\in Fp_0+Fn_0$.  %
  Using the associative and commutative laws in~$B$, we find that~$A$
  is associative.
\end{proof}

\begin{rema}
  The examples in~\ref{unaryDimAtMostFour}\ref{QtwoDim} show
  that not every associative subalgebra of a split octonion algebra is
  contained in a quaternion algebra. 
\end{rema}

\section{Associative and commutative subalgebras}

\begin{nthm}[Associativity of subalgebras]\label{assocResult}
  Let\/ $\OO$ be an octonion algebra. 
  \begin{enumerate}
  \item\label{assoc<4}%
    Every subalgebra of dimension less than~$4$ is associative. 
  \item\label{assoc=4}%
    If\/ $\OO$ is split then there do exist subalgebras of
    dimension~$4$ that are not associative, namely, the totally
    singular ones (of the form $n\OO$
    or $\OO n$ with ${n^2=0\ne n\in\OO}$). 
    Every other subalgebra of dimension~$4$ is associative.
  \item\label{assoc>4}%
    Every subalgebra of dimension greater than~$4$ contains a
    non-associative subalgebra of dimension~$4$.
  \item\label{assocSplit}%
    If\/ $\OO$ is split then a given subalgebra is not associative if,
    and only if, it contains a maximal totally singular subspace.
  \item\label{assocNotSplit}%
    If\/ $\OO$ is not split then every proper subalgebra is
    associative.
  \end{enumerate}
\end{nthm}
\begin{proof}
  If a subalgebra is not totally singular then it contains~$1$.  We
  have seen in~\ref{unaryDimAtMostFour} that each subalgebra~$A$ with
  $1\in A$ and $\dim{A}\le4$ is associative.  A totally singular
  subalgebra is associative if, and only if, it has dimension less
  than~$4$; see~\ref{allTotallySingularSubalgebras}. %
  By~\ref{allTotallySingularSubalgebras}.\ref{totSing4orbits}, the
  totally singular subalgebras of dimension~$4$ are those of the form
  $n\OO$ and~$\OO n$, repsectively, with $n^2=0\ne n$. %
  So assertions~\ref{assoc<4} and~\ref{assoc=4} are established. 

  Each proper subalgebra of dimension greater than~$4$ contains a
  totally singular subalgebra of dimension~$4$ (of the form~$n\OO$ for
  some nilpotent element~$n$, see~\ref{dimGreaterFour}). %
  These subalgebras are not associative. %
  This gives assertions~\ref{assoc>4} and~\ref{assocSplit}.

  If~$\OO$ is not split then every proper subalgebra has dimension at
  most~$4$, and there are no nilpotent elements apart from~$0$. So
  assertions~\ref{assoc<4} and~\ref{assoc=4} yield
  assertion~\ref{assocNotSplit}.
\end{proof}

\begin{coro}
  A $4$-dimensional subalgebra of\/~$\OO$ is associative if, and only
  if, it contains a neutral element for its multiplication. %
  By~\ref{neutralElement}, such a neutral element in a $4$-dimensional
  subalgebra necessarily equals~$1$. %
  \qed
\end{coro}

\begin{nthm}[Commutativity of subalgebras]\label{commResult}
  Let\/ $\OO$ be an octonion algebra, and let\/~$A$ be a subalgebra. %
  If\/~$\OO$ is split, we write $\OO = \HH+\HH w$, with
  $\HH = F^{2\times2}$ and\/~$w^2=1$. %
  The subalgebra~$A$ is commutative in the following cases:
  \begin{enumerate}
  \item\label{commDim1}%
    $\dim{A} = 1$.
  \item\label{commDim2}%
    $\dim{A} = 2$, and $A = F+Fa \cong F[X]/(X^2-\tr(a)X+\N(a))$.
  \item\label{commDim2trivial}%
    $\dim{A} = 2$, and~$A$ is in the $\Aut[F]\OO$-orbit of\/
    $(Fp_0+Fn_0)w$.
  \item\label{commDim3}%
    $\dim{A} = 3$, and $A$ is in the $\Aut[F]\OO$-orbit of\/
    $F+(Fp_0+Fn_0)w$. %
    \SaveEnumi
  \end{enumerate}
  If\/ $\Char{F} \ne 2$ then these are the only commutative
  subalgebras.
\\
  If\/ $\Char{F}=2$ we have the following additional cases (and no
  others):
  \begin{enumerate}
    \RecallEnumi%
  \item\label{commDim3char2}%
    $\dim{A} = 3$, and~$A$ is in the $\Aut[F]\OO$-orbit of the
    Heisenberg algebra ${Fn_0+(Fp_0+Fn_0)w}$.
  \item\label{commDim4sing}%
    $\dim{A} = 4$, and~$A$ is in the $\Aut[F]\OO$-orbit of\/
    $B+(Fp_0+Fn_0)w$, where $B = F+Fb$ with
    $b\in (1^\perp\cap\HH)\smallsetminus F$.
  \item\label{commDim4nonsing}%
    $\dim{A} = 4$, and~$A$ is a totally inseparable field extension of
    exponent~$1$.
  \end{enumerate}
  In any case, every commutative subalgebra of\/~$\OO$ is associative,
  as well. 
\end{nthm}

\begin{proof}
  We use our classification of subalgebras,
  see~\ref{allTotallySingularSubalgebras},
  \ref{unaryDimAtMostFour}, and~\ref{dimGreaterFour}.
  We note that $Fp_0+Fn_0$ is not commutative, in fact $p_0n_0 = n_0
  \ne 0 = n_0p_0$. So we can exclude from the list of subalgebras each
  algebra that contains a subalgebra isomorphic to $Fp_0+Fn_0$, or
  isomorphic to $\kappa(Fp_0+Fn_0) = F\gal{p_0}+Fn_0$. %
  In particular, we note that $n_0\OO$ is excluded by the observation
   $p_0 = \left(
    \begin{smallmatrix}
      1 & 0 \\
      0 & 0 
    \end{smallmatrix}
  \right)
  =
  \left(
    \begin{smallmatrix}
      0 & 1 \\
      0 & 0 
    \end{smallmatrix}
  \right)
  \left(
    \begin{smallmatrix}
      0 & 0 \\
      1 & 0 
    \end{smallmatrix}
  \right)
  \in n_0\HH \subseteq n_0\OO$. %
  This leaves us with the following list:

  \begin{itemize}
  \item Subalgebras of dimension~$1$: these are clearly commutative.
  \item Subalgebras in the orbit of $Q \coloneqq (Fp_0+Fn_0)w$: here the
    multiplication is trivial, and thus commutative.
  \item Subalgebras in the orbit of $F+Q$: these algebras are
    commutative because~$Q$ is.
  \item Heisenberg algebras (in the orbit of
    $Fn_0+Q = (n_0w)\OO\cap\OO(n_0w)$,
    see~\ref{Heisenberg}\ref{heisenbergIndeed}); the only nontrivial
    products of elements in $\{n_0,p_0w,n_0w\}$ are $n_0(p_0w) = n_0w$
    and $(p_0w)n_0 = -n_0w$.  So the Heisenberg algebras are
    commutative if, and only if, the field~$F$ has characteristic two.
  \item Subalgebras in the orbit of $F+Fn_0+Q$: such an algebra is
    commutative if, and only if, the Heisenberg algebra $Fn_0+Q$ is
    commutative, and we already know that this happens precisely if
    $\Char{F}=2$.
  \item Totally inseparable extensions fields of degree~$4$ (only
    possible if $\Char{F}=2$) are commutative.
  \item Subalgebras in the orbit of $B+Q$, where $B \subseteq \HH$ is
    either a two-dimensional composition algebra, or an inseparable
    quadratic field extension.

    For $b\in B$ and $x\in Fp_0+Fn_0$ we compute $b(xw) = (xb)w$ and
    $(xw)b = (x\gal{b})w$. %
    If $B\subseteq 1^\perp$ then $\gal{b}=b$ holds for each $b\in B$,
    and $B+Q$ is commutative. %
    If $B\not\subseteq 1^\perp$ then $B$ is a composition algebra. So
    either $B\cong F\times F$, or $B$ is a separable quadratic
    extension field. In both cases, we find $b\in B$ such $b-\gal{b}$
    is invertible, and $b(xw) = (xw)b$ implies $x=0$. %
    So the algebra $B+Q$ is not commutative in these cases.
  \end{itemize}
  Our discussion yields that the commutative subalgebras of~$\OO$ are
  exactly those given in the statement of the theorem. %
  From~\ref{assocResult} we infer that each one of these commutative
  subalgebras is in fact associative.
\end{proof}

\section{Maximal subalgebras}
\label{sec:maximal}

\begin{ndef}[The lattice of isomorphism types of subalgebras]
  \label{lattice}
  Let $p,n,m$ be elements of a (necessarily split) octonion
  algebra~$\OO$ with $p^2=p\notin\{0,1\}$, $n^2=0=m^2$, and such that
  $m,n$ are linearly independent. Moreover, assume that $pn=n$ and
  $np=0=nm$ (then $mn=0$). %
  We abbreviate $S\coloneqq F+Fp \cong F\times F$, and
  $Q\coloneqq Fm+Fn$.  Note that $S$ is a split two-dimensional
  composition algebra, while the multiplication on~$Q$ is trivial. %
  Note also that $T \coloneqq {F+Fp+Fn}$ is isomorphic to the algebra of
  upper triangular matrices in~$F^{2\times2}$, and that $n\OO\cap\OO
  n$ is a Heisenberg algebra (see~\ref{Heisenberg}).

  Let $X$ be any subalgebra of dimension~$d$ in~$\OO$. Our
  classification of subalgebras
  (see~\ref{allTotallySingularSubalgebras}, \ref{dimGreaterFour},
  and~\ref{unaryDimAtMostFour}) can be rephrased, as follows.
  \begin{description}
  \item[$d=1$:]%
    Then $X$ is in the orbit of $Y\in\{F,Fp,Fn\}$ under~$\Aut[F]\OO$.
  \item[$d=2$:]%
    Either~$X$ is in the $\Aut[F]\OO$-orbit of
    \(
      Y\in \{ S, %
      {F+Fn}, {Fn+Fp}, {Fn+F\gal{p}}, Q %
      \} \,,
    \)
    or~$X$ is a separable extension field~$E$ over~$F$, or
    $\Char{F}=2$ and $X$ is an inseparable extension field~$D$
    of~$F$. %
    Moreover, any subalgebra isomorphic to~$E$ or~$D$ is in the
    $\Aut[F]\OO$-orbit of~$E$ or~$D$, respectively.
  \item[$d=3$:]%
    Then $X$ is in the $\Aut[F]\OO$-orbit of
    \( Y\in\{ T, F+Q, {m\OO\cap\OO n}, {n\OO\cap\OO n} \} %
    \).
  \item[$d=4$:]%
    Then $X$ is in the orbit of $Y$ such that either
    $Y\in\{ {F+(n\OO\cap\OO n)}, {n\OO}, {\OO n}\}$, or one of the
    following holds:
    \begin{itemize}
    \item $Y\cong F^{2\times2}$ is a split quaternion algebra,
    \item there is a subalgebra $S\cong F\times F$ contained
      in~$Q^\perp$ such that $Y={S+Q}$,
    \item $Y$ is a quaternion field (denoted by $H$ in
      Figure~\ref{fig:lattice}),
    \item there exists a separable extension field~$E$ of degree~$2$
      contained in $Q^\perp$ such that $Y=E+Q$,
    \item $\Char{F}=2$ and there exists an inseparable extension
      field~$D$ of degree~$2$ contained in $Q^\perp$ such that
      $Y=D+Q$,
    \item $\Char{F}=2$ and $Y$ is a totally inseparable extension
      field (denoted by $K$ in Figure~\ref{fig:lattice}) of degree~$4$
      and exponent~$1$.
    \end{itemize}
  \item[$d=5$:]%
    Then $X$ is in the $\Aut[F]\OO$-orbit of~$n\OO+\OO n$.
  \item[$d=6$:]%
    Then $X$ is in the $\Aut[F]\OO$-orbit
    of~$Q^\perp = m\OO+n\OO = \OO m+\OO n$.
  \end{description}
  The algebras $F$ and $Fp$ are isomorphic but not in the same orbit
  under $\Aut[F]\OO$. %
  In any other case, two subalgebras of~$\OO$ are in the same orbit
  under $\Aut[F]\OO$ if, and only if, the subalgebras are isomorphic
  (as abstract $F$-algebras).

  Consider subalgebras $X\ne F$ and~$Y^*$ with $\dim{Y^*}>1$. %
  Since every subalgebra isomorphic to~$Y^*$ is in the
  $\Aut[F]\OO$-orbit of~$Y^*$, there exists a subalgebra $X^*$
  isomorphic to~$Y^*$ with $X\subseteq X^*$ if, and only if, the
  algebra~$Y^*$ has a subalgebra~$Y$ isomorphic to~$X$. %
  We can thus describe inclusions between subalgebras by a graph with
  the $\Aut[F]\OO$-orbits of subalgebras as vertices, and edges
  indicating embeddings. In Figure~\ref{fig:lattice} we have further
  simplified that graph by identifying orbits of subalgebras that
  contain extension fields; corresponding vertices are indicated by
  doubled boundaries at their labels (namely, $E$, $H$, $E+Q$, $D$,
  $K$, $D+Q$). Recall that the existence of such extensions heavily
  depends on the structure of~$F$ (e.g., no such extension exists
  if~$F$ is algebraically closed); the corresponding embeddings
  (indicated by dotted lines in Figure~\ref{fig:lattice}) exist only
  if the subalgebras in question exist. The inseparable extensions
  only exist if $\Char{F}=2$, this is indicated by dashed boundaries
  (and dashed edges indicating embeddings).

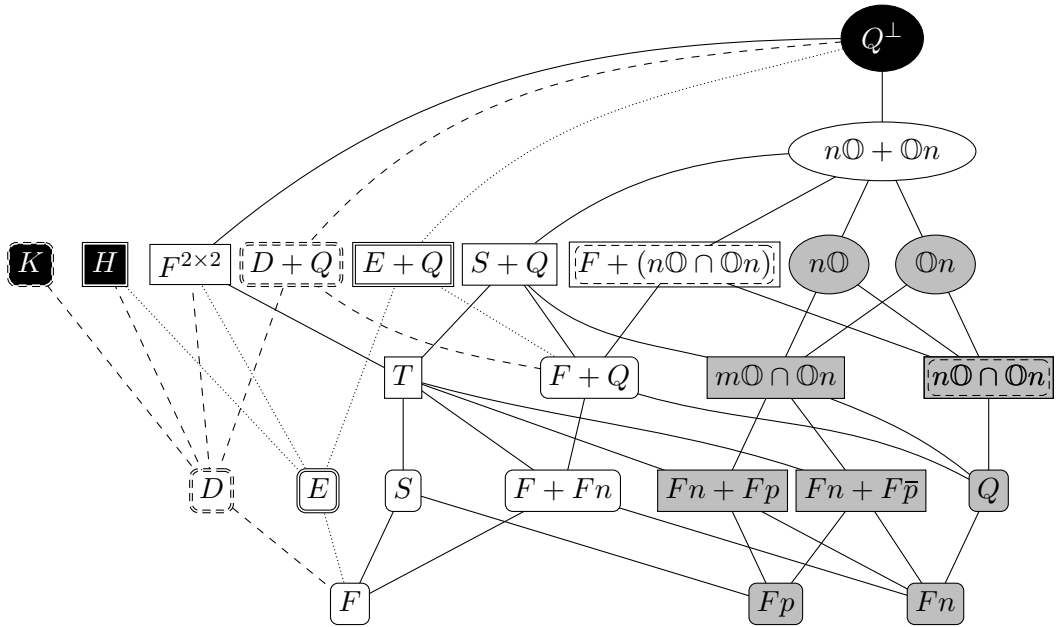
\begin{figure}[h!]
  \centering
  \begin{tikzpicture}[%
    xscale=1.4, %
    yscale=1.5, %
    every node/.append style={rectangle, fill=white, draw=black, %
        inner sep=3pt, %
        minimum size=2pt}]

    \coordinate (F) at (-5,0) {} ;%
    \coordinate (Fp) at (-1,0) {} ;%
    \coordinate (Fn) at ( .5,0) {} ;%
    \coordinate (Fn+Fp) at (-1.5,1) {} ;%
    \coordinate (Fn+Fq) at ( -.2,1) {} ;%
    \coordinate (Fn+Fnw) at (1,1) {} ;
    \coordinate (Fn+Fnw+Fpw) at (-1,2) {} ;%
    \coordinate (On*nO) at (1,2) {} ;%
    \coordinate (On) at ( .5,3) {} ;%
    \coordinate (nO) at (-.5,3) {} ;%
    \coordinate (5) at (0,4) {} ;%
    \coordinate (6) at (0,5) {} ;%

    \coordinate (E) at (-5.3,1) {} ;%
    \coordinate (D) at (-6.3,1) {} ;%
    \coordinate (H) at (-7.3,3) {} ;%
    \coordinate (K) at (-8,3) {} ;%
    
    \coordinate (F+Fn) at (-3.0,1) {} ;%
    \coordinate (FxF) at (-4.5,1) {} ;%
    \coordinate (FxF') at (-4.75,1) {} ;%
    \coordinate (F+Fn+Fnw) at (-2.75,2) {} ;%
    \coordinate (F+Fn+Fp) at (-4.5,2) {} ;%
    \coordinate (F2x2) at (-6.5,3) {} ;%

    \coordinate (F+On*nO) at (-1.95,3) {} ;%
    \coordinate (D+Q) at (-5.55,3) {} ;%
    \coordinate (E+Q) at (-4.5,3) {} ;%
    \coordinate (S+Q) at (-3.5,3) {} ;%
    \coordinate (B+Q') at (-4.6,3) {} ;%

    \draw (Fp)--(Fn+Fp) ;%
    \draw (Fn)--(Fn+Fp) ;%
    \draw (Fn)--(Fn+Fnw) ;%
    \draw (Fp)--(Fn+Fq) ;%
    \draw (Fn)--(Fn+Fq) ;%
    \draw (Fn+Fp)--(Fn+Fnw+Fpw) ;%
    \draw (Fn+Fnw)--(On*nO) ;
    \draw[in=-20,out=140] (Fn+Fnw) to (Fn+Fnw+Fpw) ;
    \draw (Fn+Fq)--(Fn+Fnw+Fpw) ;%
    \draw (On*nO)--(nO) ;%
    \draw (On*nO)--(On) ;%
    \draw (Fn+Fnw+Fpw)--(nO) ;%
    \draw (Fn+Fnw+Fpw)--(On) ;%
    \draw (nO)--(5) ;%
    \draw (On)--(5) ;%
    \draw (5)--(6) ;%

    \draw[dashed] (F)--(D) ;%
    \draw[densely dotted] (F)--(E) ;%
    \draw[dashed] (D)--(K) ;%
    \draw[dashed] (D)--(H) ;%
    \draw[densely dotted] (E)--(H) ;%

    \draw (F)--(F+Fn) ;%
    \draw (Fn)--(F+Fn) ;%
    \draw (Fp)--(FxF) ;%
    \draw (FxF)--(F+Fn+Fp) ;%
    \draw (Fn+Fp)--(F+Fn+Fp) ;%
    \draw[in=-20,out=150] (Fn+Fnw) to (F+Fn+Fnw) ;
    \draw[in=-15,out=160] (Fn+Fq) to (F+Fn+Fp) ;%

    \draw (F)--(FxF) ;%
    \draw (F+Fn+Fp) to (F2x2) ;%
    \draw[densely dotted] (E) to (F2x2) ;%
    \draw[dashed] (D)--(F2x2) ;%

    \draw (On*nO)--(F+On*nO) ;%
    \draw (F+On*nO)--(5) ;%

    \draw[densely dotted] (E) to (E+Q) ;%
    \draw[dashed] (D) to (D+Q) ;%
    \draw (F+Fn+Fnw) to (S+Q) ;%
    \draw[in=-40,out=170,dashed] (F+Fn+Fnw) to (D+Q) ;%
    \draw[densely dotted] (F+Fn+Fnw) to (E+Q) ;%
    \draw (F+Fn+Fp) to (S+Q) ;%
    \draw (F+Fn+Fnw) to (F+On*nO) ;%
    \draw[in=-50,out=160] (Fn+Fnw+Fpw) to (S+Q) ;%
    \draw (F+Fn) to (F+Fn+Fp) ;%
    \draw (F+Fn) to (F+Fn+Fnw) ;%

    \draw[in=-180,out=40] (F2x2) to (6) ;%
    \draw[in=-175,out=50,dashed] (D+Q) to (6) ;%
    \draw[in=-165,out=50, densely dotted] (E+Q) to (6) ;%
    \draw[in=-180,out=40] (S+Q) to (5) ;%
    
    \node[rounded corners=1mm] at (F) {$F\vphantom{p}$} ;%
    \node[rounded corners=1mm][fill=lightgray] at (Fp) {$Fp$} ;%
    \node[rounded corners=1mm][fill=lightgray] at (Fn) {$Fn\vphantom{p}$} ;%
    \node[fill=lightgray] at (Fn+Fp) {$Fn+Fp$} ;%
    \node[rounded corners=1mm][fill=lightgray] at (Fn+Fnw) {$Q\vphantom{p}$} ;%
    \node[fill=lightgray] at (Fn+Fq) {$Fn+F\gal{p}$} ;%
    \node[fill=lightgray] at (Fn+Fnw+Fpw) {$m\OO\cap\OO n\vphantom{p}$} ;
    \node[fill=lightgray] at (On*nO) {$n\OO\cap\OO n\vphantom{p}$} ;%
    \node[fill=none,rounded corners=1mm,densely dashed] at (On*nO) {$\phantom{nO\cap on}$} ;%
    \node[fill=none] at (On*nO) {$n\OO\cap\OO n\vphantom{p}$} ;%
    \node[fill=lightgray,style=ellipse] at (nO) {$n\OO\vphantom{p}$} ;%
    \node[fill=lightgray,style=ellipse] at (On) {$\OO n\vphantom{p}$} ;%
    \node[style=ellipse] at (5) {$n\OO+\OO n\vphantom{p}$} ;%
    \node[fill=black,style=ellipse] at (6) {$\color{white}Q^\perp$} ;%

    \node[rounded corners=1mm,double,densely dashed,double] at (D) {$D\vphantom{p}$} ;%
    \node[rounded corners=1mm,double] at (E) {$E\vphantom{p}$} ;%
    \node[fill=black,rounded corners=1mm,double,densely dashed,double] at (K) {$\color{white}K\vphantom{p}$} ;%
    \node[fill=black,double] at (H) {$\color{white}H\vphantom{p}$} ;%
    
    \node[rounded corners=1mm] at (FxF) {$S\vphantom{p}$} ;%
    \node[rounded corners=1mm] at (F+Fn) {$F+Fn\vphantom{p}$} ;%
    \node at (F+Fn+Fp) {$T\vphantom{p}$} ;%
    \node[rounded corners=1mm] at (F+Fn+Fnw) {$F+Q\vphantom{p}$} ;%
    \node at (F2x2) {$F^{2\times2}$} ;%

    \node[rounded corners=1mm,densely dashed] at (F+On*nO) {\phantom{$\,F+nO\cap O n\,$} } ;%
    \node[fill=none] at (F+On*nO) {$F+(n\OO\cap\OO n)\vphantom{p}$ } ;%
    \node[rounded corners=1mm,double,densely dashed] at (D+Q) {$D+Q\vphantom{p}$ } ;%
    \node[double] at (E+Q) {$E+Q\vphantom{p}$ } ;%
    \node at (S+Q) {$S+Q\vphantom{p}$ } ;%

  \end{tikzpicture}
  \caption{The lattice of orbits of subalgebras, see~\ref{lattice}.}
  \label{fig:lattice}
\end{figure}

In Figure~\ref{fig:lattice}, the maximal element~$\OO$ is omitted.
The maximal proper subalgebras are in the orbits with black labels,
i.e., those represented by non-split quaternion subalgebras, totally
inseparable extensions of degree~$4$ and exponent~$1$,
and~$Q^\perp$. Grey labels indicate subalgebras that do not
contain~$1$ (i.e., totally singular ones), and rectangular labels (as
opposed to elliptical ones) indicate associative algebras. %
Rectangular labels with rounded corners indicate commutative algebras
(all of them are associative).  Double lining is used for subalgebras
whose existence depends on the ground field, and dashed lines indicate
subalgebras that only exist if $\Char{F}=2$. %
Finally, the algebras $n\OO\cap\OO n$ and $F+(n\OO\cap\OO n)$ are
commutative precisely if $\Char{F}=2$. %

Figures~\ref{fig:latticeCharOdd} and~\ref{fig:latticeRestricted} give
the simpler graphs for the cases where $\Char{F}\ne2$, and where
$\Char{F}\ne2$ and $F$ allows no quadratic field extension. %
If~$F$ is finite then the vertices labeled by ``$K$'' or by ``$H$''
have to be omitted from the graph in Figure~\ref{fig:lattice}. %
\end{ndef}

\begin{rems}
  Let~$\OO$ be the split octonion algebra over~$F$; we know that~$\OO$
  is unique up to isomorphism. %

  If there does exist a quadratic extension field (separable or not)
  over~$F$ then the split quaternion algebra~$F^{2\times2}$ over~$F$
  contains subalgebras isomorphic to that extension
  field. Doubling~$F^{2\times2}$ we obtain a split octonion algebra,
  and thus see that~$\OO$ contains subalgebras isomorphic to any
  quadratic extension field of~$F$.

  If there exists a quaternion field~$H$ over~$F$ then a suitable
  double of~$H$ is a split octonion algebra, and we obtain that~$\OO$
  contains subalgebras isomorphic to~$H$. However, using the Kaplansky
  radical (cp.~\cite[Ch.\,XII, Prop.\,6.1, pp.\,450\,f]{MR2104929})
  one sees that there do exist fields~$F$ with quadratic extensions
  of~$F$ that are not contained in any quaternion field over~$F$ ---
  although there do exist quaternion fields over~$F$ (see the examples
  in~\cite[p.\,461]{MR2104929} or~\cite[2.4]{MR3188856}). %
  If~$F$ is such a field, then the split octonion algebra over~$F$
  contains quaternion fields and also quadratic extensions that are
  not contained in any quaternion field, so the inclusion between~$E$
  and~$H$ in Figure~\ref{fig:lattice} has to be interpreted with due
  care. 
    
  Finally, assume that there exists a totally inseparable extension
  field~$K$ of degree~$4$ and exponent~$1$ over~$F$.  Then~\cite[4.3,
  4.6]{MR2746044} shows that there is a split octonion algebra
  containing a subalgebra isomorphic to~$K$.
\end{rems}

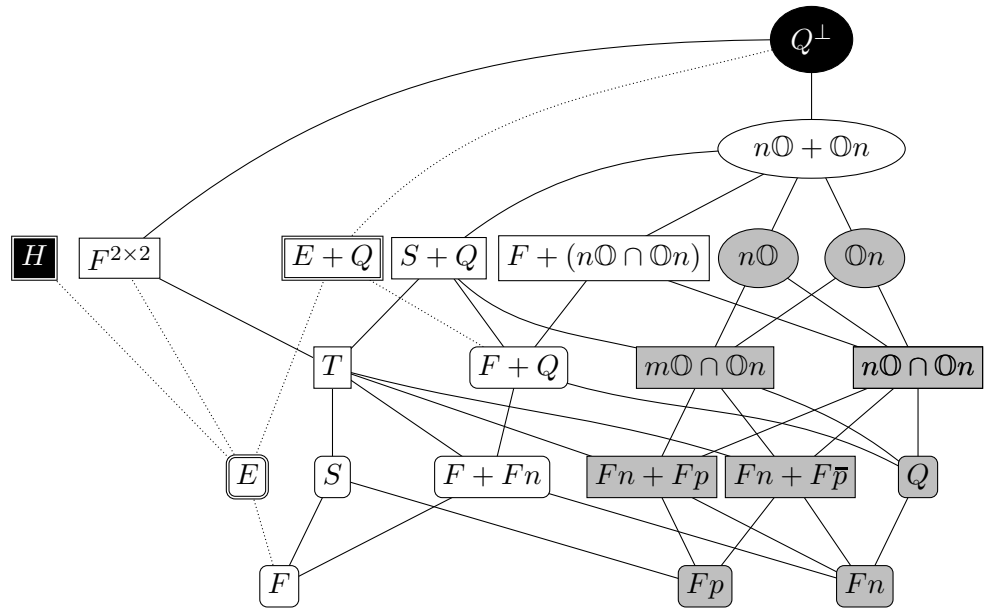
\begin{figure}[h!]
  \centering
  \begin{tikzpicture}[%
    xscale=1.4, %
    yscale=1.45, %
    every node/.append style={rectangle, fill=white, draw=black, %
        inner sep=3pt, %
        minimum size=2pt}]

    \coordinate (F) at (-5,0) {} ;%
    \coordinate (Fp) at (-1,0) {} ;%
    \coordinate (Fn) at ( .5,0) {} ;%
    \coordinate (Fn+Fp) at (-1.5,1) {} ;%
    \coordinate (Fn+Fq) at ( -.2,1) {} ;%
    \coordinate (Fn+Fnw) at (1,1) {} ;
    \coordinate (Fn+Fnw+Fpw) at (-1,2) {} ;%
    \coordinate (On*nO) at (1,2) {} ;%
    \coordinate (On) at ( .5,3) {} ;%
    \coordinate (nO) at (-.5,3) {} ;%
    \coordinate (5) at (0,4) {} ;%
    \coordinate (6) at (0,5) {} ;%

    \coordinate (E) at (-5.3,1) {} ;%
    \coordinate (D) at (-6.3,1) {} ;%
    \coordinate (H) at (-7.3,3) {} ;%
    \coordinate (K) at (-8,3) {} ;%
    
    \coordinate (F+Fn) at (-3.0,1) {} ;%
    \coordinate (FxF) at (-4.5,1) {} ;%
    \coordinate (FxF') at (-4.75,1) {} ;%
    \coordinate (F+Fn+Fnw) at (-2.75,2) {} ;%
    \coordinate (F+Fn+Fp) at (-4.5,2) {} ;%
    \coordinate (F2x2) at (-6.5,3) {} ;%

    \coordinate (F+On*nO) at (-1.95,3) {} ;%
    \coordinate (D+Q) at (-5.55,3) {} ;%
    \coordinate (E+Q) at (-4.5,3) {} ;%
    \coordinate (S+Q) at (-3.5,3) {} ;%
    \coordinate (B+Q') at (-4.6,3) {} ;%

    \draw (Fp)--(Fn+Fp) ;%
    \draw (Fn)--(Fn+Fp) ;%
    \draw (Fn)--(Fn+Fnw) ;%
    \draw (Fp)--(Fn+Fq) ;%
    \draw (Fn)--(Fn+Fq) ;%
    \draw (Fn+Fp)--(On*nO) ;%
    \draw (Fn+Fp)--(Fn+Fnw+Fpw) ;%
    \draw (Fn+Fnw)--(On*nO) ;%
    \draw[in=-20,out=140] (Fn+Fnw) to (Fn+Fnw+Fpw) ;
    \draw (Fn+Fq)--(On*nO) ;%
    \draw (Fn+Fq)--(Fn+Fnw+Fpw) ;%
    \draw (On*nO)--(nO) ;%
    \draw (On*nO)--(On) ;%
    \draw (Fn+Fnw+Fpw)--(nO) ;%
    \draw (Fn+Fnw+Fpw)--(On) ;%
    \draw (nO)--(5) ;%
    \draw (On)--(5) ;%
    \draw (5)--(6) ;%

    \draw[densely dotted] (F)--(E) ;%
    \draw[densely dotted] (E)--(H) ;%

    \draw (F)--(F+Fn) ;%
    \draw (Fn)--(F+Fn) ;%
    \draw (Fp)--(FxF) ;%
    \draw (FxF)--(F+Fn+Fp) ;%
    \draw (Fn+Fp)--(F+Fn+Fp) ;%
    \draw[in=-20,out=150] (Fn+Fnw) to (F+Fn+Fnw) ;
    \draw[in=-15,out=160] (Fn+Fq) to (F+Fn+Fp) ;%

    \draw (F)--(FxF) ;%
    \draw (F+Fn+Fp) to (F2x2) ;%
    \draw[densely dotted] (E) to (F2x2) ;%

    \draw (On*nO)--(F+On*nO) ;%
    \draw (F+On*nO)--(5) ;%

    \draw[densely dotted] (E) to (E+Q) ;%
    \draw (F+Fn+Fnw) to (S+Q) ;%
    \draw[densely dotted] (F+Fn+Fnw) to (E+Q) ;%
    \draw (F+Fn+Fp) to (S+Q) ;%
    \draw (F+Fn+Fnw) to (F+On*nO) ;%
    \draw[in=-50,out=160] (Fn+Fnw+Fpw) to (S+Q) ;%
    \draw (F+Fn) to (F+Fn+Fp) ;%
    \draw (F+Fn) to (F+Fn+Fnw) ;%

    \draw[in=-180,out=40] (F2x2) to (6) ;%
    \draw[in=-165,out=50, densely dotted] (E+Q) to (6) ;%
    \draw[in=-180,out=40] (S+Q) to (5) ;%
    
    \node[rounded corners=1mm] at (F) {$F\vphantom{p}$} ;%
    \node[rounded corners=1mm][fill=lightgray] at (Fp) {$Fp$} ;%
    \node[rounded corners=1mm][fill=lightgray] at (Fn) {$Fn\vphantom{p}$} ;%
    \node[fill=lightgray] at (Fn+Fp) {$Fn+Fp$} ;%
    \node[rounded corners=1mm][fill=lightgray] at (Fn+Fnw) {$Q\vphantom{p}$} ;%
    \node[fill=lightgray] at (Fn+Fq) {$Fn+F\gal{p}$} ;%
    \node[fill=lightgray] at (Fn+Fnw+Fpw) {$m\OO\cap\OO n\vphantom{p}$} ;%
    \node[fill=lightgray] at (On*nO) {$n\OO\cap\OO n\vphantom{p}$} ;%
    \node[fill=none] at (On*nO) {$n\OO\cap\OO n\vphantom{p}$} ;%
    \node[fill=lightgray,style=ellipse] at (nO) {$n\OO\vphantom{p}$} ;%
    \node[fill=lightgray,style=ellipse] at (On) {$\OO n\vphantom{p}$} ;%
    \node[style=ellipse] at (5) {$n\OO+\OO n\vphantom{p}$} ;%
    \node[fill=black,style=ellipse] at (6) {$\color{white}Q^\perp$} ;%

    \node[rounded corners=1mm,double] at (E) {$E\vphantom{p}$} ;%
    \node[fill=black,double] at (H) {$\color{white}H\vphantom{p}$} ;%
    
    \node[rounded corners=1mm] at (FxF) {$S\vphantom{p}$} ;%
    \node[rounded corners=1mm] at (F+Fn) {$F+Fn\vphantom{p}$} ;%
    \node at (F+Fn+Fp) {$T\vphantom{p}$} ;%
    \node[rounded corners=1mm] at (F+Fn+Fnw) {$F+Q\vphantom{p}$} ;%
    \node at (F2x2) {$F^{2\times2}$} ;%

    \node at (F+On*nO) {$F+(n\OO\cap\OO n)\vphantom{p}$ } ;%
    \node[double] at (E+Q) {$E+Q\vphantom{p}$ } ;%
    \node at (S+Q) {$S+Q\vphantom{p}$ } ;%
  \end{tikzpicture}
  \caption{The lattice of orbits of subalgebras for the case where
    $\Char{F}\ne2$, see~\ref{lattice}.}
  \label{fig:latticeCharOdd}
\end{figure}

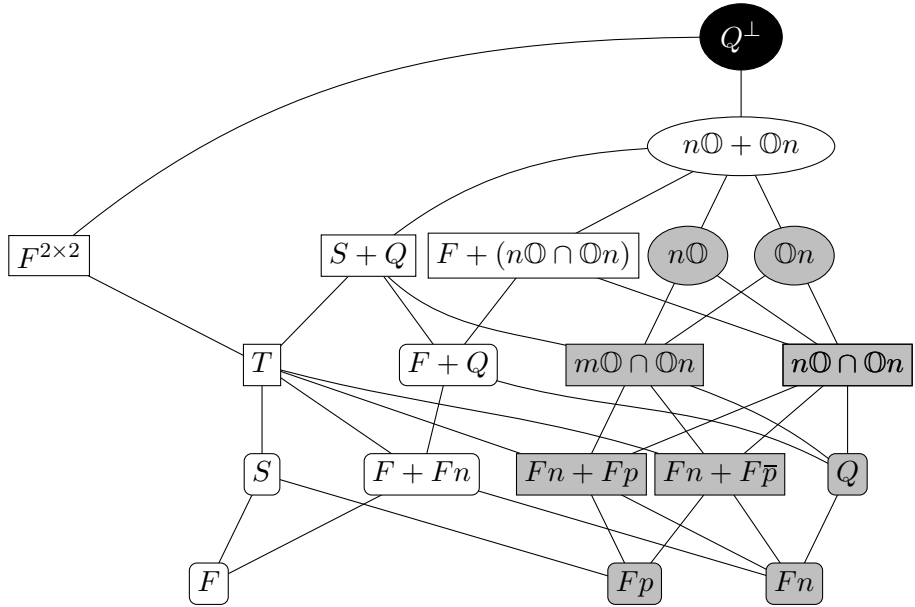
\begin{figure}[h!]
  \centering
  \begin{tikzpicture}[%
    xscale=1.4, %
    yscale=1.45, %
    every node/.append style={rectangle, fill=white, draw=black, %
        inner sep=3pt, %
        minimum size=2pt}]

    \coordinate (F) at (-5,0) {} ;%
    \coordinate (Fp) at (-1,0) {} ;%
    \coordinate (Fn) at ( .5,0) {} ;%
    \coordinate (Fn+Fp) at (-1.5,1) {} ;%
    \coordinate (Fn+Fq) at ( -.2,1) {} ;%
    \coordinate (Fn+Fnw) at (1,1) {} ;
    \coordinate (Fn+Fnw+Fpw) at (-1,2) {} ;%
    \coordinate (On*nO) at (1,2) {} ;%
    \coordinate (On) at ( .5,3) {} ;%
    \coordinate (nO) at (-.5,3) {} ;%
    \coordinate (5) at (0,4) {} ;%
    \coordinate (6) at (0,5) {} ;%

    \coordinate (E) at (-5.3,1) {} ;%
    \coordinate (D) at (-6.3,1) {} ;%
    \coordinate (H) at (-7.3,3) {} ;%
    \coordinate (K) at (-8,3) {} ;%
    
    \coordinate (F+Fn) at (-3.0,1) {} ;%
    \coordinate (FxF) at (-4.5,1) {} ;%
    \coordinate (FxF') at (-4.75,1) {} ;%
    \coordinate (F+Fn+Fnw) at (-2.75,2) {} ;%
    \coordinate (F+Fn+Fp) at (-4.5,2) {} ;%
    \coordinate (F2x2) at (-6.5,3) {} ;%

    \coordinate (F+On*nO) at (-1.95,3) {} ;%
    \coordinate (D+Q) at (-5.55,3) {} ;%
    \coordinate (E+Q) at (-4.5,3) {} ;%
    \coordinate (S+Q) at (-3.5,3) {} ;%
    \coordinate (B+Q') at (-4.6,3) {} ;%

    \draw (Fp)--(Fn+Fp) ;%
    \draw (Fn)--(Fn+Fp) ;%
    \draw (Fn)--(Fn+Fnw) ;%
    \draw (Fp)--(Fn+Fq) ;%
    \draw (Fn)--(Fn+Fq) ;%
    \draw (Fn+Fp)--(On*nO) ;%
    \draw (Fn+Fp)--(Fn+Fnw+Fpw) ;%
    \draw (Fn+Fnw)--(On*nO) ;%
    \draw[in=-20,out=140] (Fn+Fnw) to (Fn+Fnw+Fpw) ;
    \draw (Fn+Fq)--(On*nO) ;%
    \draw (Fn+Fq)--(Fn+Fnw+Fpw) ;%
    \draw (On*nO)--(nO) ;%
    \draw (On*nO)--(On) ;%
    \draw (Fn+Fnw+Fpw)--(nO) ;%
    \draw (Fn+Fnw+Fpw)--(On) ;%
    \draw (nO)--(5) ;%
    \draw (On)--(5) ;%
    \draw (5)--(6) ;%

    \draw (F)--(F+Fn) ;%
    \draw (Fn)--(F+Fn) ;%
    \draw (Fp)--(FxF) ;%
    \draw (FxF)--(F+Fn+Fp) ;%
    \draw (Fn+Fp)--(F+Fn+Fp) ;%
    \draw[in=-20,out=150] (Fn+Fnw) to (F+Fn+Fnw) ;
    \draw[in=-15,out=160] (Fn+Fq) to (F+Fn+Fp) ;%

    \draw (F)--(FxF) ;%
    \draw (F+Fn+Fp) to (F2x2) ;%

    \draw (On*nO)--(F+On*nO) ;%
    \draw (F+On*nO)--(5) ;%

    \draw (F+Fn+Fnw) to (S+Q) ;%
    \draw (F+Fn+Fp) to (S+Q) ;%
    \draw (F+Fn+Fnw) to (F+On*nO) ;%
    \draw[in=-50,out=160] (Fn+Fnw+Fpw) to (S+Q) ;%
    \draw (F+Fn) to (F+Fn+Fp) ;%
    \draw (F+Fn) to (F+Fn+Fnw) ;%

    \draw[in=-180,out=40] (F2x2) to (6) ;%
    \draw[in=-180,out=40] (S+Q) to (5) ;%
    
    \node[rounded corners=1mm] at (F) {$F\vphantom{p}$} ;%
    \node[rounded corners=1mm][fill=lightgray] at (Fp) {$Fp$} ;%
    \node[rounded corners=1mm][fill=lightgray] at (Fn) {$Fn\vphantom{p}$} ;%
    \node[fill=lightgray] at (Fn+Fp) {$Fn+Fp$} ;%
    \node[rounded corners=1mm][fill=lightgray] at (Fn+Fnw) {$Q\vphantom{p}$} ;%
    \node[fill=lightgray] at (Fn+Fq) {$Fn+F\gal{p}$} ;%
    \node[fill=lightgray] at (Fn+Fnw+Fpw) {$m\OO\cap\OO n\vphantom{p}$} ;
    \node[fill=lightgray] at (On*nO) {$n\OO\cap\OO n\vphantom{p}$} ;%
    \node[fill=none] at (On*nO) {$n\OO\cap\OO n\vphantom{p}$} ;%
    \node[fill=lightgray,style=ellipse] at (nO) {$n\OO\vphantom{p}$} ;%
    \node[fill=lightgray,style=ellipse] at (On) {$\OO n\vphantom{p}$} ;%
    \node[style=ellipse] at (5) {$n\OO+\OO n\vphantom{p}$} ;%
    \node[fill=black,style=ellipse] at (6) {$\color{white}Q^\perp$} ;%
    
    \node[rounded corners=1mm] at (FxF) {$S\vphantom{p}$} ;%
    \node[rounded corners=1mm] at (F+Fn) {$F+Fn\vphantom{p}$} ;%
    \node at (F+Fn+Fp) {$T\vphantom{p}$} ;%
    \node[rounded corners=1mm] at (F+Fn+Fnw) {$F+Q\vphantom{p}$} ;%
    \node at (F2x2) {$F^{2\times2}$} ;%

    \node at (F+On*nO) {$F+(n\OO\cap\OO n)\vphantom{p}$ } ;%
    \node at (S+Q) {$S+Q\vphantom{p}$ } ;%

  \end{tikzpicture}
  \caption{The lattice of orbits of subalgebras for the case where $F$
    is quadratically closed and $\Char{F}\ne2$, see~\ref{lattice}.}
  \label{fig:latticeRestricted}
\end{figure}

\begin{ndef}[Open Problems]
  For each subalgebra~$X$ in a given octonion algebra~$\OO$ over~$F$:
  \begin{enumerate}
  \item Determine the full group $\Aut{X}$ of all $\ZZ$-linear
    automorphisms of~$X$, and the group~$\Aut[F]X$ of all algebra
    automorphisms of~$X$.
  \item Determine the stabilizers $\Aut\OO_X$ and $\Aut[F]\OO_X$
    of~$X$ in~$\Aut\OO$ and in~$\Aut[F]\OO$, repsectively.
  \item Determine the subgroups of $\Aut{X}$ induced by $\Aut\OO_X$
    and by $\Aut[F]\OO_X$, respectively.
  \end{enumerate}
\end{ndef}



\newpage

\medskip

\begin{small}
  \begin{minipage}[t]{0.3\linewidth}
    Norbert Knarr, \\ 
    Markus J. Stroppel 
  \end{minipage}
  \begin{minipage}[t]{0.6\linewidth}
    LExMath\\
    Fakult\"at 8\\
    Universit\"at Stuttgart\\
    70550 Stuttgart\\ 
    stroppel@mathematik.uni-stuttgart.de 
  \end{minipage}
\end{small}

\end{document}